\documentclass[11pt,reqno]{amsart}

\usepackage[T1]{fontenc}
\usepackage[utf8]{inputenc}
\usepackage{csquotes}
\usepackage[english]{babel}

\usepackage[title,titletoc,page]{appendix}
\usepackage[a4paper,left=3cm,right=3cm,top=2cm,bottom=4cm,bindingoffset=5mm]{geometry}
\usepackage[citecolor=black,colorlinks=true,linkcolor=black,urlcolor=black,linktocpage=true]{hyperref}
\usepackage[all]{hypcap}
\usepackage{setspace}
\usepackage{xcolor}

\usepackage{aligned-overset}
\usepackage{amsfonts}
\usepackage{amsmath}
\usepackage{amssymb}
\usepackage{amsthm}
\usepackage{dsfont}
\usepackage{float}
\usepackage{mathrsfs}
\usepackage{mathtools}
\mathtoolsset{showonlyrefs}
\usepackage{bigints}
\usepackage{thmtools}
\usepackage[shortlabels]{enumitem}

\usepackage{graphicx}

\usepackage{mathscinet}
\usepackage[backend=biber,style=alphabetic,giveninits=true,isbn=false,url=false]{biblatex}
\usepackage{version}
\usepackage[obeyFinal]{todonotes}
\usepackage{cancel}
\newtheorem{theorem}{Theorem}[section]
\newtheorem{lemma}[theorem]{Lemma}
\newtheorem{proposition}[theorem]{Proposition}
\newtheorem{corollary}[theorem]{Corollary}

\theoremstyle{definition}
\newtheorem{definition}[theorem]{Definition}
\newtheorem{assumption}[theorem]{Assumption}
\newtheorem{setting}[theorem]{Setting}

\theoremstyle{remark}
\newtheorem{remark}[theorem]{Remark}

\numberwithin{equation}{section}

\def\eps{\varepsilon}

\def\R{\mathbb R}

\def\R{{{\mathbb R}}}

\def\res{R_{\lambda,h}}
\def\rest{\mathbf{v}_{\kappa,h}}

\newcommand{\dd}{ \mathrm{d}}

\newcommand{\intbar}{\etaathop{\int\etaakebox(-13.5,0){\rule[4pt]{.7em}{0.3pt}}
\kern-6pt}\nolimits}

\newcommand{\be}{\begin{equation}}
\newcommand{\ee}{\end{equation}}
\newcommand{\bea}{\begin{equation*}}
\newcommand{\eea}{\end{equation*}}

\newcommand{\beq}{\begin{equation}}
\newcommand{\eeq}{\end{equation}}

\DeclareMathOperator*{\argmin}{arg\,min}

\usetikzlibrary{backgrounds}
\usetikzlibrary{decorations.pathmorphing,backgrounds,fit,calc,through}
\usetikzlibrary{arrows}
\usetikzlibrary{shapes,decorations,shadows}
\usetikzlibrary{fadings}
\usetikzlibrary{patterns}
\usetikzlibrary{mindmap}
\usetikzlibrary{decorations.text}
\usetikzlibrary{decorations.shapes}

\usepackage{mathtools,bbm,bm}
\usepackage[normalem]{ulem}
\usepackage{cancel}

\newcommand{\cA}{\mathcal{A}}

\newcommand{\cC}{\mathcal{C}}
\newcommand{\cD}{\mathcal{D}}

\newcommand{\cH}{\mathcal{H}}

\newcommand{\cL}{\mathcal{L}}
\newcommand{\cM}{\mathcal{M}}

\newcommand{\cO}{\mathcal{O}}

\newcommand{\cU}{\mathcal{U}}
\newcommand{\cV}{\mathcal{V}}

\newcommand{\cX}{\mathcal{X}}
\newcommand{\cY}{\mathcal{Y}}
\newcommand{\cZ}{\mathcal{Z}}

\newcommand{\bR}{\mathbb{R}}

\newcommand{\bfH}{\mathbf{H}}

\newcommand{\bONE}{\mathbbm{1}}
\newcommand{\vn}[1]{\left|  \left| #1\right|  \right|} 
\newcommand{\ssup}[1]{\left\lceil #1 \right\rceil}
\newcommand{\iinf}[1]{\left\lfloor #1 \right\rfloor}
\newcommand{\ip}[2]{\left \langle #1,#2 \right \rangle}

\newcommand{\ri}[1]{{\color{teal}#1}}

\title[Existence of viscosity solutions]{Existence of viscosity solutions for Hamilton-Jacobi equations via Lyapunov control}
\author{Serena Della Corte}
\address{Delft Institute of Applied Mathematics, Delft University of Technology, Mekelweg 4, 2628 CD Delft, The Netherlands.}
\email{s.dellacorte@tudelft.nl} 

\author{Richard C. Kraaij}
\address{Delft Institute of Applied Mathematics, Delft University of Technology, Mekelweg 4, 2628 CD Delft, The Netherlands.}
\email{r.c.kraaij@tudelft.nl}
\thanks{This work was funded by The Netherlands Organisation for Scientific Research (NWO), grant number 613.009.148}
\date{\today}
\begin{document}

\begin{abstract}
    We give a new perspective on the existence of viscosity solutions for a stationary and a time-dependent first-order Hamilton-Jacobi equation. Following recent comparison principles, we work in a framework in which we consider a subsolution and a supersolution for two equations in terms of two Hamiltonians that can be seen as an upper semi-continuous upper and lower semi-continuous lower bound of our original Hamiltonian respectively. The bounds are made rigorous in terms of Youngs inequality.

    The bounds are furthermore formulated in a way that incorporate a Lyapunov function which allows us to restrict part of the analysis to compact sets and to work with almost optimizers of the considered control problems. For this reason, we can relax assumptions on the control problem: most notably, we do not need completeness of set of controlled paths. Moreover, this strategy avoids a-priori analysis on the regularity of the candidate solutions.

    To complete our picture, we exhibit our result in two contexts. First, we consider Riemannian manifolds with smooth boundary, in which the dynamics allows for both "inward" and "outward" drift. The boundary conditions are embedded into the Hamiltonians itself. Second, we consider the solutions to a Fillipov differential equation, i.e. one with discontinuous vector field. We show that our notion of Hamiltonians leverages the natural embedding of the discontinuity in an associated set-valued differential inclusion.

    \smallskip
\noindent\textbf{Keywords}\emph{:  Hamilton--Jacobi equations, Lyapunov function, viscosity solutions, optimal control theory, Dynamic Programming principle, boundary conditions, Fillipov differential equation} 
\end{abstract}
\maketitle

\section{Introduction}

In this work, we present a novel perspective on the existence of viscosity solutions for both stationary and time-dependent first-order Hamilton-Jacobi equations on a d-dimensional smooth manifold $\cM$. Let $\cH: T^*\cM \to \R$. The specific equations we address are:
\begin{equation}\label{eq:stationary-HJB}
    u(x) - \lambda \cH(x,\dd u(x)) = h(x),
\end{equation}
where $\lambda> 0$ and $h : \cM \rightarrow \bR$, and its evolutionary version on $\cM \times [0,T]$:
\begin{gather}\label{eq:time-dep-HJB}
    \begin{cases}
        \partial_t u(x,t) + \kappa u(x,t) - \cH(x, \dd_x u(\cdot,t)(x)) = 0, & \text{if $t>0$,}\\
        u(x,0) = h(x) & \text{if $t=0$}
    \end{cases}
\end{gather}
with $\kappa \geq 0$. For $\cL : T\cM \rightarrow (-\infty,\infty]$, we consider the typical value functions
    \begin{equation}
        \label{def:eq:value_function:intro}
        R_{\lambda,h}(x) = \sup_{\gamma \in {Adm}, \gamma(0)=x} \int_0^\infty  e^{-\lambda^{-1}t} \left(\frac{h(\gamma(t))}{\lambda} -  \cL(\gamma(t),\dot{\gamma}(t)) \right) \, dt,
    \end{equation}
    and
    \begin{equation}
        \label{def:eq:value_function_timedep:intro}
        \rest(x,t) = \sup_{\gamma \in Adm, \gamma(0)=x} \int_0^t - e^{-\kappa s} \cL(\gamma(s), \dot{\gamma}(s))\, ds + e^{-\kappa t} h(\gamma(t)).
    \end{equation}
The new perspective introduced here builds on the “upper–lower” Hamiltonian framework developed in the works of Tataru \cite{Ta92,Ta94}, Crandall–Lions \cite{CrLi94}, and Feng and co-authors \cite{Fe06,FK06,FeMiZi21,AmFe14}, and further developed in \cite{FeKa09,DFL11,KrSc21,DCKr23,DCKr24}. These works established comparison principles for viscosity solutions under minimal regularity by introducing upper and lower envelopes of the Hamiltonian. In particular, our approach is closely aligned with the Lyapunov-based techniques of \cite{FeKa09,KrSc21,DCKr23,DCKr24}, where the Lyapunov function plays a key role in controlling growth and compactness.

\textbf{Main ideas.}
The novelty of our approach is based on a combination of insights:
\begin{enumerate}[(a)]
    \item Working on unbounded domains, it is of use to embed a Lyapunov function $\Upsilon$ into the domain of $\cH$, and replaces $\cH$ by a upper Hamiltonian $\cH_\dagger$ and lower Hamiltonian $\cH_\ddagger$ in \eqref{eq:stationary-HJB} and \eqref{eq:time-dep-HJB} for sub- and supersolutions respectively. 
    \item Working on domains with boundaries, one similarly resorts to upper semi-continuous lower and lower semi-continuous upper Hamiltonians to deal with discontinuities at the boundary.
\end{enumerate}

Thus, taking the value functions in \eqref{def:eq:value_function:intro} and \eqref{def:eq:value_function_timedep:intro}, we essentially work with the following four assumptions:

\begin{enumerate}[(a)]
    \item \, [Assumption \ref{assumption:Lagrangian}] We assume that the Lagrangian $\cL$ is globally bounded from below, the map $x \mapsto \inf_{v \in T_x\cM} \cL(x,v)$ is globally bounded from above, and has super-linear growth in $v$. 
\end{enumerate}

We have a Lyapunov function to control how far curves can travel on a fixed budget

\begin{enumerate}[(a),resume]
    \item \, [Assumption \ref{assumption:Lyapunov}] There exists a function $\Upsilon$ with compact sublevelsets and a constant $C_\Upsilon$ such that for any curve $\gamma$ and $T > 0$ we have
\begin{equation*}
    \Upsilon(\gamma(T)) - \Upsilon(\gamma(0)) \leq \int_0^T \cL(\gamma(t),\dot{\gamma}(t)) \, dt + T C_\Upsilon.
\end{equation*}
\end{enumerate}

Together these two sets of assumptions allow us to establish various regularity properties of the value functions. To code the value functions in terms of Hamilton-Jacobi equations, we furthermore assume the existence of a upper and lower Hamiltonian $\cH_\dagger$ and $\cH_\ddagger$ such that 
\begin{enumerate}[(a),resume]
    \item \, [Assumption \ref{assumption:Hdagger}] $\cH_\dagger$ is upper semi-continuous and $\cL$ and $\cH_\dagger$ together satisfy Youngs upper bound:
\begin{equation*}
    \ip{v}{p} \leq \cL(x,v) + \cH_{\dagger}(x,p)
\end{equation*}
\item \, [Assumption \ref{assumption:Hddagger}] $\cL$ and $\cH_\ddagger$ satisfy a integrated version of Youngs lower bound: for any $x \in \cM$ and $f \in C^1(\cM)$ there exists a curve $\gamma$ with $\gamma(0) = x$ such that
\begin{equation*}
    f(\gamma(T)) - f(\gamma(0)) \geq \int_0^T \cL(\gamma(t),\dot{\gamma}(t)) + \cH_\ddagger (\gamma(t), \dd f(\gamma(t))) \, dt.
\end{equation*}
\end{enumerate}

Under these four assumptions, we will show that \eqref{def:eq:value_function:intro} and \eqref{def:eq:value_function_timedep:intro} are viscosity sub- and supersolutions for \eqref{eq:stationary-HJB} and \eqref{eq:time-dep-HJB} in terms of two operators $H_\dagger$ and $H_\ddagger$ defined as
 \begin{align}
     & H_\dagger f_\dagger(x) := (1-\eps)\cH_\dagger(x,\dd f(x)) + \eps C_\Upsilon, \\
     & H_\ddagger f_\ddagger(x) := (1+\eps)\cH_\ddagger(x,\dd f(x)) - \eps C_\Upsilon,
 \end{align}
where $\varepsilon \in (0,1)$ and
\begin{align}
     & f_\dagger := (1-\eps)f(x) + \eps \Upsilon, \\
     & f_\ddagger := (1+\eps)f(x) - \eps \Upsilon,
\end{align}
 for two collections of appropriately chosen $f$.

\smallskip

To show applicability of our framework in new contexts, we consider two distinct contexts.

\textit{Application 1:} We start with the setting of smooth Riemannian manifolds with a $C^1$ boundary as studied in Euclidean context in \cite{So86,BaLi91,Da06}. Due to the expository nature of this example, we restrict ourselves to a context where the Lagrangian orthogonally decomposes in a normal and tangential direction.  

As a key difference, we note that our definition of viscosity solutions does not distinguish between boundary or in the interior points directly. Rather, the bounds are now framed in terms of $\cH_\dagger$ and $\cH_\ddagger$, which possibly behave differently for boundary points. This difference is then furthermore dictated by semi-continuity properties, and the two sides of Youngs inequalities, see Assumptions  \ref{assumption:Hdagger} and \ref{assumption:Hddagger}. 

To establish that our framework is tight enough for uniqueness, we give a comparison principle in the context of a geodesically convex domain. Our approach easily extends to the context with isolated boundaries as in \cite{DuIiSo90} which includes more complex settings with non-orthogonally splitting Lagrangians.

\textit{Application 2:} We consider the value functions corresponding to the solution to a Fillipov differential equation \cite{Fi13}, i.e. a differential equation with discontinuous vector field. Solutions are usually formulated in terms of differential inclusions based for a set-valued map which captures the behaviour of the vector field "close-by". Our Hamiltonians $\cH_\dagger,\cH_\ddagger$ leverage the upper-hemi-continuity of the set-valued map.


\textbf{Discussion and relation to previous works.} Our method thus extends the line of research initiated by \cite{Ta92,Ta94,CrLi94,Fe06,FK06,FeKa09,DFL11,FeMiZi21,KrSc21,DCKr23,DCKr24}, while differing from classical approaches such as Perron’s method \cite{Is87,IsLi90,CIL92,CuDL07}. In contrast to methods based on regularity of the Hamiltonian or completeness of control sets, in our approach:
\begin{itemize}
    \item The use of $\Upsilon$ allows us to work with approximating sequences in the definition of sub- and supersolutions. In combination with super-linear growth for $\cL$, we can work with almost-optimizing curves in the definition of the value functions. This combination allows us to avoid a-priori analysis on the regularity of the value functions.
    \item The use of upper semi-continuous $\cH_\dagger$ and lower semi-continuous $\cH_\ddagger$ and the relation to $\cL$ via the two sides of Youngs inequality allows us to directly embed e.g. boundary conditions into $\cH_\dagger$ and $\cH_\ddagger$.
\end{itemize}

Furthermore, we do not assume convexity of the Hamiltonian, allowing treatment of Hamilton–Jacobi–Isaacs equations where $\cH$ is given by “sup–inf” or “inf–sup” of convex operators.

Finally, even though our work is inspired by the strategy in \cite[Chapter 8]{FK06}, our method diverges from it in several key aspects. While \cite{FK06} relies on completeness and pseudo-resolvent properties of the value function, our approach—based on Lyapunov control and the operators $H_\dagger, H_\ddagger$ applies uniformly to both stationary and evolutionary equations.

\textbf{Organization of the work.}
Our work is structured as follows: we start in Section \ref{section:generalsetting} with preliminaries. The main results, Theorems \ref{theorem:stationary} and \ref{theorem:timedep}, as well as complete versions of the four above mentioned assumptions \ref{assumption:Lagrangian}, \ref{assumption:Lyapunov}, \ref{assumption:Hdagger} and \ref{assumption:Hddagger} are stated in Section \ref{sec:mainresults}. In Section \ref{section:Riemmannian_manifold_withBoundary} we exhibit our result in the context of a smooth Riemannian manifold with boundary. In Section \ref{section:Fillipov} we consider Fillipov vector fields.


Next follow four technical proof sections.  In Section \ref{section:dpp} we establish the dynamic programming principle. In Section \ref{sec:boundedness_solutions} we establish the boundedness of the candidate solutions.
In Section \ref{section:consequences_Lyapunov}, we consider what the Lyapunov function $\Upsilon$ brings us in terms of reach of trajectories with controlled costs, as well as the allowable unboundedness of solutions. In Section \ref{section:tightness} we consider tightness of small-time paths on the basis of the superlinearity of $\cL$.

In Sections \ref{section:subsolutions} and \ref{section:supersolutions} we finally establish the sub- and supersolution properties respectively.

\smallskip

\textbf{Acknowledgements}
RK thanks Louis-Pierre Chaintron for discussions that lead to the results in Sections \ref{section:Riemmannian_manifold_withBoundary} and \ref{section:Fillipov}. This research was supported by the Dutch Research Council NWO grantnumber 613.009.148.

\section{General setting and main definitions}
\label{section:generalsetting}

In this section, we give some basic notions and definitions used throughout the paper.

\subsection{Topological preliminaries}

Throughout the paper, $\cX$ will be a d-dimensional smooth manifold. $T\cX := \bigsqcup_{x\in\cX} T_x \cX$ and $T^*\cX := \bigsqcup_{x\in\cX} T^*_x \cX$ are the \textit{tangent bundle} and \textit{cotangent bundle} over $\cX$ respectively. We will write $C(\cX)$ and $C^1(\cX)$ for the continuous and continuously differentiable functions on $\cX$. We refer to e.g. \cite{Tu2010} for more details about smooth manifolds.

Our Hamilton-Jacobi equations will generally be based on  a subset $\cM \subseteq \cX$. Set 
\begin{align*}
    T\cM & := \left\{(x,v) \in T\cX\, \middle| \, x \in \cM \right\}, \\
    T^*\cM & := \left\{(x,p) \in T^*\cX \, \middle| \, x \in \cM \right\}. 
\end{align*}

To make sure we have a proper notion of continuously differentiable functions on $\cM$, we introduce the following notion.

\begin{definition} \label{definition:wellBehaved}
    We say $\cM$ is a well-behaved subspace of $\cX$ if for any $f,g \in C^1(\cX)$ such that $f = g$ on $\cM$ it holds that $df = dg$ on $\cM$. Suppose that $\cM$ is a well-behaved subspace of $\cX$, we then denote
    \begin{equation*}
        C^1(\cM) := \left\{f \in C(\cM) \, \middle| \, \exists \, \hat{f} \in C^1(\cX): \hat{f}|_\cM = f \right\}
    \end{equation*}
    and for $f \in C^1(\cM)$ we write $df(x) = d\hat{f}(x)$ for any $\hat{f} \in C^1(\cX)$ such that $\hat{f}|_\cM = f$.
\end{definition}

Throughout the paper we will assume that $\cM$ is well-behaved. An equivalent characterization is as follows.

\begin{lemma}
    Suppose that $\cM \subseteq \cX$ is contained in the $\cX$ closure of the $\cX$ interior of $\cM$. Then $\cM$ is a well-behaved subspace of $\cX$.
\end{lemma}

\begin{proof}
Let $f,g \in C^1(\cX)$ such that $f = g$ on $\cM$. It follows that $df = dg$ on the interior $\text{int} \cM$ of $\cM$. As $df$ and $dg$ are continuous on $\cX$, it follows that $df = dg$ on the closure of $\text{int} \cM$ establishing the claim.
\end{proof}

On $\cM$ we furthermore consider $C_b(\cM)$, $C_l(\cM)$ $C_u(\cM)$, and $C_{cc}(\cM)$ the spaces of continuous functions that are bounded, bounded below, bounded above and constant outside of a compact set respectively. We denote $C^1(\cM), C_{cc}^1(\cM),C_l^1(\cM), C_u^1(\cM)$ the corresponding subspaces obtained by intersecting with $C^1(\cM)$.

For any function $f : \cM \rightarrow \bR$ denote by $f^*$ and $f_*$ its upper and lower semi-continuous regularization respectively. 
 
\smallskip

We introduce the notions of \textit{viscosity solutions} for an Hamilton-Jacobi equation $f-\lambda Af=h$ and for the time-dependent version $\partial_t f +\lambda f -  A f = 0$ in the Appendix \ref{appendix:viscosity}

\smallskip

In Sections \ref{section:Riemmannian_manifold_withBoundary} and \ref{section:Fillipov} we will work with set valued maps and differential inclusions. We will denote these as follows, for topological spaces $\cX,\cY$ a set-valued map $F : \cX \rightrightarrows \cY$ maps $x \in \cX$ to a non-empty set $F(x) \subseteq \cY$. Regularity properties are treated in Appendix \ref{appendix:differential_inclusions}, and a more thorough analysis can be found in \cite{AlBo06}.

\section{Assumptions and main results}\label{sec:mainresults}

In this section, we give our main results, namely Theorems \ref{theorem:stationary} and \ref{theorem:timedep}. We start out with setting up the Lagrangian and our candidate solutions.

\subsection{The dynamic control problem}

For the finite time-horizon problem, as well as the discounted problem, we define our value-functions in terms of a Lagrangian $\cL : T\cM \rightarrow (-\infty,\infty]$ integrated over admissible paths. Whereas typically the full set of absolutely continuous paths will be admissible, we coin admissibility, so that in applications we can choose to work e.g. with the subset of piecewise smooth trajectories.

\begin{definition} \label{definition:curves}
We say that $Adm \subseteq AC ([0,\infty),\cM)$ is an admissible set of curves if the following three properties hold:
    \begin{enumerate}[(a)]
        \item \label{item:adm_exist} For any $x \in \cM$ there exits $\gamma \in Adm$ with $\gamma(0) = x$,
        \item \label{item:adm_t} If $\{\gamma(t)\}_{t\geq 0} \in \cC$ and $\tau >0$, then $\{\gamma(t+\tau)\}_{t\geq 0} \in \cC$;
        \item \label{item:adm_concatenate} If $\gamma_1, \gamma_2 \in \cC$ and $\tau >0$ and let $\gamma$ be the curve defined as
        \begin{equation}
        \gamma(t) =
            \begin{cases}
                \gamma_1(t) &t\leq \tau \\
                \gamma_2(t-\tau) & t>\tau.
            \end{cases}
        \end{equation}
        Then, $\{\gamma(t)\}_{t\geq 0} \in \cC$.
    \end{enumerate}
\end{definition}

\begin{assumption} \label{assumption:Lagrangian}
$\cM \subseteq \cX$ and $Adm$ are well behaved and admissible in the sense of Definitions  \ref{definition:wellBehaved} and \ref{definition:curves} respectively. The map $\cL : T\cM \rightarrow (-\infty,\infty]$ is measurable and satisfies:
\begin{enumerate}[(a)]
\item  \label{item:assumption:Lagrangian:lowerbound} We have
  \begin{equation}
    \inf_{(x,v) \in T \cM} \cL(x,v) =: \cL_{\min} > - \infty.
  \end{equation}
  \item \label{item:assumption:Lagrangian:upperbound} There is a constant $\cL_{\max}$ such that for any $x \in \cM$ there exists a curve $\gamma \in Adm$ with $\gamma(0) = x$ and for any continuous $\phi : \bR^+ \rightarrow \bR^+$ and $T > 0$
  \begin{equation*}
      \int_0^T  \cL(\gamma(t),\dot{\gamma}(t)) \phi(t) \, dt \leq \cL_{\max} \int_0^T \phi(t) \, dt.
  \end{equation*}
    \item \label{item:assumption:Lagrangian:superlinear_growth} For all compact sets $K \subseteq \cM$ the function $\cL$ has superlinear growth:
    \begin{equation}
    \lim_{N\to\infty} \inf_{x\in K} \inf_{\substack{v \in T_x \cM \\ |v| \geq N}} \frac{\cL(x,v)}{1+|v|} = + \infty.
\end{equation}
\end{enumerate}
\end{assumption}

\begin{remark}
    Item \ref{item:assumption:Lagrangian:upperbound} essentially translates to
    \begin{equation*}
        \sup_x \inf_{v \in T_x \cM} \cL(x,v) < \infty.
    \end{equation*}
    This bound, however, does not immediately imply that there are curves in $Adm$ that reproduce this uniform upper bound. We thus assume this in terms of the existence of appropriate curves.
\end{remark}

In Lemma \ref{lemma:bounds_on_L} in Appendix \ref{appendix:lagrangian} we give conditions on the Hamiltonian to show the upper and lower bounds \ref{item:assumption:Lagrangian:lowerbound} and \ref{item:assumption:Lagrangian:upperbound}. Moreover, we establish in Section \ref{section:tightness} that Assumption \ref{assumption:Lagrangian} implies local equi-continuity of sets of curves for which we uniformly control the Lagrangian. In terms of $Adm$ and $\cL$ we can introduce the value functions.

\begin{definition} \label{definition:solutions}
    Let $\cM$, $Adm$ and $\cL$ be as in Assumption \ref{assumption:Lagrangian}. 
    
    For a starting function $h \in C_b(\cM)$ we define the candidate solutions $R_{\lambda}: \cM \to \bR$ for $\lambda > 0$ and $v_{\kappa,h}: \cM \times [0,T] \to \bR$ for $\kappa \geq 0$ through the optimization problems 
\begin{equation} \label{def:eq:value_function}
	R_{\lambda,h}(x) := \sup_{\gamma \in {Adm}, \gamma(0)=x} J_{\lambda,h}(\gamma),
\end{equation}
and
\begin{equation}\label{def:eq:value_function_timedep}
     v_{\kappa,h}(x,t) := \sup_{\gamma \in Adm, \gamma(0)=x} K_{\kappa,h}(\gamma,t), 
\end{equation}
where 
\begin{align*}
    J_{\lambda,h}(\gamma) & := \int_0^\infty e^{-\lambda^{-1}t} \left(\frac{h(\gamma(t))}{\lambda} - \cL(\gamma(t),\dot{\gamma}(t)) \right) \,  \dd t, \\
    K_{\kappa,h}(\gamma,t) & := e^{-\kappa t} h(\gamma(t)) - \int_0^t e^{-\kappa s} \cL(\gamma(s), \dot{\gamma}(s))\, \dd s.
\end{align*} 
We say that $\gamma$ is $q \geq 0$ admissible for $R_{\lambda,h}(x)$ or $v_{\kappa,h}(x,t)$ if $\gamma(0) = x$ and
\begin{equation} \label{eqn:q_admissible}
    R_{\lambda,h}(x) \leq J_{\lambda,h}(\gamma) + q, \qquad v_{\kappa,h}(x,t) \leq K_{\kappa,h}(\gamma,t) + q
\end{equation}
respectively.
\end{definition}

In Section \ref{section:dpp}, we establish that the admissibility of $Adm$ of Assumption \ref{assumption:Lagrangian} implies the \textit{dynamic programming principle (DPP)} for both $v_{\kappa,h}$ and $R_{\lambda,h}$. For the latter, we observe that, on the basis of Assumption \ref{assumption:Lagrangian} and Lemma \ref{lemma:curvesIBP}, an integration by parts argument allows us the rewrite
\begin{equation} \label{eqn:IPB_in_resolvent}
    J_{\lambda,h}(\gamma) := \int_0^\infty \lambda^{-1} e^{-\lambda^{-1}t} \left(h(\gamma(t)) - \int_0^t \cL(\gamma(s),\dot{\gamma}(s)) \,\dd s \right) \,  \dd t.
\end{equation}
This formulation shows that both optimization problems are essentially defined over the same controlled dynamics.

\smallskip

As a final note, and to cross over to the next section, we note that Assumption \ref{assumption:Lagrangian} \ref{item:assumption:Lagrangian:lowerbound} and \ref{item:assumption:Lagrangian:upperbound} together imply boundedness for the value functions, if initialized with bounded $h$.

\subsection{Controlling the reach of controlled paths}

Our next assumption ensures that curves with controlled Lagrangian cost do not travel out too far. This is formalized using a containment function $\Upsilon$. 

\begin{definition}[Containment function]\label{def:lyapunovfunction}
  We call $\Upsilon : \cM \rightarrow [1,\infty)$ a \emph{containment function} if
    \begin{enumerate}[(a)]
        \item $\inf_{x \in \cM} \Upsilon(x) = 1$,
        \item for every $c \geq 1$ the set $\{y \, | \, \Upsilon(y) \leq c\}$ is compact. 
    \end{enumerate}
\end{definition}

\begin{assumption} \label{assumption:Lyapunov}
 There exists a containment function $\Upsilon$ as in Definition \ref{def:lyapunovfunction} and a constant $C_\Upsilon$ such that for all $\gamma\in Adm$ and $T>0$ the following holds
    \begin{equation}
    \Upsilon(\gamma(T)) - \Upsilon(\gamma(0)) \leq  \int_0^T \cL(\gamma(t),\dot{\gamma}(t))  \, dt + T C_\Upsilon.
    \end{equation}
\end{assumption}

We explore the verification of Assumption \ref{assumption:Lyapunov} in Section \ref{section:verificationLyapunov} below.

\subsection{Assumptions on the upper and lower Hamiltonians} \label{sec:assumptions} 

In the next two sets of assumptions, we introduce $H_\dagger$ and $H_\ddagger$ as well-behaved dual upper and lower bounds for $\cL$. Immediately after, in Definition \ref{definiton:HdaggerHddagger}, we introduce the operators $H^\Upsilon_\dagger$ and $H^\Upsilon_\ddagger$ that combine $H_\dagger, H_\ddagger$ with the control of admissible curves obtained from the use of the containment function $\Upsilon$ introduced in Assumption \ref{assumption:Lyapunov}.

\begin{assumption} \label{assumption:Hdagger}
    Consider $\cH_\dagger: T^*\cM \to \bR$ and let $H_\dagger \subseteq C^1_l(\cM) \times M(\cM)$ be an operator such that $H_\dagger f(x) = \cH_\dagger(x, \dd f(x))$. The following properties hold: 
	\begin{enumerate}[(a)]
		\item \label{item:assumptionHdagger:usc} 
		The map $(x,p) \mapsto \cH_\dagger(x,p)$ is upper semi-continuous.
        \item \label{item:assumptionHdagger:Young_upperbound} For all $x \in \cM$, $v \in T_x \cM$ and $p \in T_x^* \cM$ we have
        \begin{equation*}
            \ip{p}{v} \leq \cH_\dagger(x,p) + \cL(x,v).
        \end{equation*}
	\end{enumerate}
\end{assumption}



\begin{assumption} \label{assumption:Hddagger}
    Consider $\cH_\ddagger: T^*\cM \to \bR$ and let $H_\ddagger \subseteq C^1_u(\cM) \times M(\cM)$ be an operator such that $H_\ddagger f(x) = \cH_\ddagger(x, \dd f(x))$. The following properties hold:  
	\begin{enumerate}[(a)]
		\item \label{item:assumptionHddagger:lsc}
		The map $(x,p) \mapsto \cH_\ddagger(x,p)$ is lower semi-continuous.
        \item \label{item:assumptionHddagger:Young_lowerbound} 
        For all $f \in \cD(H_\ddagger)$ and $x \in \cM$, there exists $\gamma \in Adm$ with $\gamma(0) = x$ such that, for any $T>0$
        \begin{equation*}
            \int_0^T \cL(\gamma(t),\dot{\gamma}(t)) \, dt \leq f(\gamma(T)) - f(\gamma(0)) - \int_0^T \cH_\ddagger(\gamma(t),\dd f(\gamma(t))) \, dt.
        \end{equation*}
        \item \label{item:assumptionHddagger:lowerBound}
        For all $f \in \cD(H_\ddagger)$ we have
        \begin{equation*}
            C_f := - \inf_{x \in \cM} \cH_\ddagger(x, \dd f(x)) < \infty.
        \end{equation*}
	\end{enumerate}
\end{assumption}

\begin{remark} \label{remark:construction_curves_convex}
The construction leading to the first inequality of \ref{item:assumptionHddagger:Young_lowerbound} is well-known in the context where $p \mapsto \cH_\ddagger(x,p)$ is convex, and is related to solving the differential inclusion 
\begin{equation}
    \dot{\gamma}(t) \in \partial_p \cH_\ddagger(\gamma(t),\dd f(\gamma(t)))
\end{equation}
(see e.g. \cite{CaSi04} or \cite{Ro70}). We will carry out a variant of this construction in the context of a Riemannian manifold with $C^1$ boundary in Section \ref{section:Riemmannian_manifold_withBoundary}.
\end{remark}

We now combine $H_\dagger$ and $H_\ddagger$ with the use of the containment function $\Upsilon$ introduced in Assumption \ref{assumption:Lyapunov}. Recall the constant $C_\Upsilon$ introduced therein.

\begin{definition}[The operators $H^\Upsilon_\dagger$ and $H^\Upsilon_\ddagger$] \label{definiton:HdaggerHddagger}
	For $f \in D(H_\dagger)$ and $\varepsilon \in (0,1)$  set 
	\begin{gather*}
		f^\varepsilon_\dagger := (1-\varepsilon) f + \varepsilon \Upsilon \\
		g_\dagger^\eps(x) := (1-\varepsilon) \cH_\dagger(x,\dd f(x)) + \varepsilon C_\Upsilon.
	\end{gather*}
	and set
	\begin{equation*}
		H^\Upsilon_\dagger := \left\{(f^\varepsilon_\dagger, g_\dagger^\eps) \, \middle| \, f \in D(H_\dagger), \varepsilon \in (0,1) \right\}.
	\end{equation*} 
	For $f \in D(H_\ddagger)$ and $\varepsilon \in (0,1)$ set 
	\begin{gather*}
		f^\varepsilon_\ddagger := (1+\varepsilon) f - \varepsilon \Upsilon \\
		g_\ddagger^\eps(x) := (1+\varepsilon) \cH_\ddagger(x,\dd f(x)) - \varepsilon C_\Upsilon.
	\end{gather*}
	and set
	\begin{equation*}
		H^\Upsilon_\ddagger := \left\{(f^\varepsilon_\ddagger,g_\ddagger^\varepsilon) \, \middle| \, f \in D(H_\ddagger), \varepsilon \in (0,1) \right\}.
	\end{equation*} 
\end{definition}

\subsection{Main results: existence of viscosity solutions}

In what follows, we present the main results concerning the existence of viscosity subsolutions and supersolutions, first for the stationary equation and subsequently for its evolutionary counterpart.

\begin{theorem}[Viscosity solution for the stationary equation]\label{theorem:stationary}
Let Assumptions \ref{assumption:Lagrangian} and \ref{assumption:Lyapunov} be satisfied. For $\lambda >0$ and $h \in C_b(\cM)$, define $\res$ as in \eqref{def:eq:value_function} and let $H_\dagger^\Upsilon$ and $H_\ddagger^\Upsilon$ be as in Definition \ref{definiton:HdaggerHddagger}.
\begin{enumerate}[(a)]
    \item Let Assumption \ref{assumption:Hdagger} be satisfied and let $h$ be upper semi-continuous. Then $(\res)^*$ is a subsolution of $u- \lambda H^\Upsilon_\dagger u = h$.
    \item  Let Assumption \ref{assumption:Hddagger} be satisfied and let $h$ be lower semi-continuous. Then $(\res)_*$ is a supersolution of $u- \lambda H^\Upsilon_\ddagger u = h$.
\end{enumerate}
\end{theorem}

\begin{theorem}[Viscosity solution for the evolutionary equation]
\label{theorem:timedep}
   Let Assumptions \ref{assumption:Lagrangian} and \ref{assumption:Lyapunov} be satisfied. For $T > 0$, $\kappa \geq 0$ and $h \in C_b(\cM)$, define $\rest(x,t):\cM \times [0,T] \to \bR$ as \eqref{def:eq:value_function_timedep} and let $H_\dagger^\Upsilon$ and $H_\ddagger^\Upsilon$ be as in Definition \ref{definiton:HdaggerHddagger}.
   \begin{enumerate}[(a)]
       \item Let Assumption \ref{assumption:Hdagger} be satisfied and let $h$ be upper semi-continuous. Then $(\rest)^*$ is a viscosity subsolution of $\partial_t u(x,t) + \kappa u(x,t) - H^\Upsilon_\dagger u(\cdot,t)(x) = 0$ with initial value $u(\cdot , 0) = h$.
       \item  Let Assumption \ref{assumption:Hddagger} be satisfied and let $h$ be lower semi-continuous. Then $(\rest)_*$ is a viscosity supersolution of $\partial_t u(x,t) + \kappa u(x,t) - H^\Upsilon_\ddagger u(\cdot,t)(x) = 0$ with initial value $u(\cdot, 0) = h$.
   \end{enumerate}
\end{theorem}

\begin{remark}
    In both Theorems \ref{theorem:stationary} and \ref{theorem:timedep} if the space $\cM$ is compact, the proof similarly yields existence for the operators $H_\dagger$ and $H_\ddagger$ instead of  $H_\dagger^\Upsilon$ and $H_\ddagger^\Upsilon$.
\end{remark}

\begin{remark}
    Similar results can be obtained for optimization problems in terms of infima by inverting time in the Hamiltonians. For example if $R_{\lambda,-h}$ is a subsolution for $f - \lambda _\dagger H_\dagger^\Upsilon f = -h$, then $\widehat{R}_{\lambda,h} := - R_{\lambda,-h}$ is a supersolution for $f + \lambda \widehat{H}_\dagger^\Upsilon f = h$, where
    \begin{equation*}
        \widehat{H}_\dagger^\Upsilon := \left\{ (f,g) \, \middle| \, (-f,g) \in H_\dagger^\Upsilon \right\}.
    \end{equation*}
    We have not explored whether the inverted time setting can be more appropriately settled with differently chosen convex combinations with $\Upsilon$ in candidate definitions for $\widehat{H}_\dagger^\Upsilon$ and $\widehat{H}_\ddagger^\Upsilon$.
\end{remark}

\section{Riemannian manifold with boundary} \label{section:Riemmannian_manifold_withBoundary}

In this section, we will consider a Hamilton-Jacobi equation on a properly defined subset $\cM$ of a $d$-dimensional complete Riemannian manifold $\cX$.

We start out by considering a Lagrangian and associated Hamiltonian on the larger space $\cX$, after which we down-select admissible curves to the smaller space $\cM$.

\begin{setting} \label{setting:subspace}
    \begin{enumerate}[(a)]
        \item \label{setting:Lagrangian} Let $\cX$ be a $d$-dimensional complete Riemannian manifold on which we have a lower semi-continuous Lagrangian $\cL_0 : T\cX \rightarrow [0,\infty]$ and $Adm(\cX)$ the space of absolutely continuous curves in $\cX$ satisfying Assumption \ref{assumption:Lagrangian}. Furthermore
    \begin{equation*}
        \cL_0(x,v) = \sup_{p \in T_x^* \cX} \ip{p}{v} - \cH_0(x,p)
    \end{equation*}
    for a continuous map $\cH_0 : T^* \cX \rightarrow \bR$ that is convex in the fibers and satisfies $\cH_0(x,0) = 0$ for all $x \in \cX$.
    \item \label{setting:Upsilon} There exists a function $\Upsilon \in C^1(\cX)$ with compact sublevel sets in $\cM$ satisfying
    \begin{equation*}
        \sup_{x \in \cM} \cH_0(x, \dd \Upsilon(x)) < \infty.
    \end{equation*}
    \item Let $g : \cX \rightarrow \cM$ be smooth, and let
    \begin{equation*}
        \cM := \left\{x \in \cX \, \middle| \, g(x) \leq  0 \right\}, \qquad \partial \cM := \left\{x \in \cX \, \middle| \, g(x) = 0 \right\}
    \end{equation*}
    and assume that $\cM$ is connected and $g$ is such that if $x \in \cX$ is such that $g(x) = 0$, then $\nabla g(x) \neq 0$. 
    \item $Adm$ is the set of absolutely continuous curves contained in $\cM$.
    \end{enumerate}
\end{setting}

Set
\begin{equation*}
    n(x) = \frac{\nabla g(x)}{|\nabla g(x)|}
\end{equation*}
for the outward normal of $\cM$ in $\cX$.

To restrict our control problem in $\cX$ to one in $\cM$, we need to redefine our Lagrangian. For an open neighbourhood $\cO$ of $\partial \cM$ in which $|\nabla g| \neq 0$, we introduce a splitting of the tangent space in terms of the direction $-n(x)$ orthogonal to the boundary, and the complement. 

Thus, for $x \in \cO$, define $P_\perp, P_\parallel : T \cM \rightarrow T \cM$ by
\begin{equation*}
    P_\perp(x,v) = \left(x, \ip{v}{-n(x)} v\right), \qquad P_\parallel(x,v) = \left(x, v - \ip{v}{-n(x)} v\right)
\end{equation*}
Note that for fixed $x \in \cM$ these are the orthogonal projections of $v \in T_x \cM$ onto the space spanned by $-n(x)$ and its complement. Define $P_\perp, P_\parallel : T^* \cM \rightarrow T^* \cM$ analogously.

\begin{assumption} \label{assumption:splittingH}
   For $x$ in a open neighbourhood $\cO$ of $\partial \cM$, we have
    \begin{equation*}
        \cL_0(x,v) = \cL_\perp(P_\perp(x,v)) +  \cL_\parallel (P_\parallel(x,v)).
    \end{equation*}
    for two lower semi-continuous maps $\cL_\perp, \cL_\parallel : T\cO \rightarrow [0,\infty]$ satisfying $\cL_\perp(x,v) = \infty$ if $P_\perp(x,v) \neq (x,v)$ and $\cL_\parallel(x,v) = \infty$ if $P_\parallel(x,v) \neq (x,v)$.
\end{assumption}
As $\cL$ is convex, this corresponds to
\begin{equation*}
    \cH_0(x,p) = \cH_\perp(x,p) + \cH_\parallel(x,p)
\end{equation*}
for maps $\cH_\perp(x,p) = \cH_\perp(P_\perp(x,p))$ and $\cH_\parallel(x,p) = \cH_\parallel(P_{\parallel}(x,p))$.

\smallskip

\begin{definition} \label{definition:adaptedBoundaryL}
    Define $\cL : T\cM \rightarrow [0,\infty]$ by
    \begin{equation*}
        \cL(x,v) = \begin{cases}
        \infty & \text{if } x \in \partial \cM, \text{ and } \ip{v}{-n(x)} < 0, \\
        \cL_{\parallel}(x,v) & \text{if } x \in \partial \cM, \ip{v}{-n(x)} = 0, \\
        & \qquad \text{ and }  \forall \, \hat{\nu} : \, \cL_0(x,\hat{\nu}) = 0: \ip{\hat{\nu}}{-n(x)} < 0, \\
        \cL_{0}(x,v) & \text{if } x \in \partial \cM, \ip{v}{-n(x)} = 0, \\
        & \qquad \text{ and } \exists \, \hat{\nu} \text{ with } \ip{\hat{\nu}}{-n(x)} \geq 0: \, \cL_0(x,\hat{\nu}) = 0, \\
        \cL_{0}(x,v) & \text{if } x \in \partial \cM, \ip{v}{-n(x)} > 0, \\
        \cL_0(x,v) & \text{if } x \notin \partial \cM.
        \end{cases} 
    \end{equation*}
    
\end{definition}

\subsection{The existence result}

To state our existence result, we start by introducing $\cH_\dagger$ and $\cH_\dagger$. These are both defined in terms of $\cH_0$, but take into account the changed nature of the behavior at the boundary.

Both definitions need to take into account two major factors: they need to verify the correct side of Young's (in)equality, and they need to have the correct semi-continuity properties to match up behaviours of the interior and boundary.

For $\cH_\ddagger$, we will make use of Appendix \ref{section:splitting_convex} to identify the proper Legendre dual of $\cL$ introduced in Definition \ref{definition:adaptedBoundaryL} and then use a maximum between the result and $\cH_0$.

In particular, the following Definition arises from Proposition \ref{proposition:identify_multid_H} , which in turn hinges on the splitting of Assumption \ref{assumption:splittingH}.

\begin{definition}\label{definition:boundaryHdagger}
Set $H_\dagger$ with $\cD(H_\dagger) = C_l^1(\cM)$ by $H_\dagger f(x) = \cH_\dagger (x, \dd f(x))$ with $\cH_\dagger : T^* \cM \rightarrow \bR$ given by
    \begin{equation*}
        \cH_\dagger(x,p) = \cH_\parallel(x,p) + 
        \begin{cases}
            \cH_{\perp}(x,p) & \text{if } x \notin \partial \cM, \\
            \cH_{\perp}(x,p) & \text{if } x \in \partial \cM \text{ and } \ip{p}{-n(x)} \leq 0, \\
            0 \vee \cH_{\perp}(x,p) & \text{if } x \in \partial \cM \text{ and } \ip{p}{-n(x)} \geq 0.
        \end{cases}
    \end{equation*}
\end{definition}

\begin{remark} \label{remark:lowerboundH_dagger_example}
    Note that $\cH_\dagger \geq \cH_0$ and $\cH_\dagger(x,p) = \cH_0(x,p)$ for $x \notin \partial \cM$. 
\end{remark}

For the introduction of $\cH_\ddagger$, we need to take into account that we need to be able to satisfies Youngs equality in Assumption \ref{assumption:Hddagger}. In standard settings, this goes back to being able to solve for the Hamiltonian flow $\dot{x} = \partial_p \cH(x,\dd f(x))$. Here, we need to take into account the boundary.

These results will be based on Appendix \ref{appendix:differential_inclusions}, which are stated for subsets of $\bR^d$, but which can be applied in our setting by local trivialization of the (co)-tangentbundle. Most importantly: any element of $\partial_p \cH_0(x,p)$ needs to be projected to only point "inward". We set this up using the map $P^+$. To interface with Appendix \ref{appendix:differential_inclusions} and give a proper definition of $\cH_\ddagger$, we also introduce a map $\Psi$ that combines projected and non-projected information.

\smallskip

Set $P^+ : T \cM \rightarrow T\cM$ by
\begin{equation*}
    P^+(x,v) = \begin{cases}
        (x,v) & \text{if } x \notin \partial \cM, \\
        (x,v) & \text{if } x \in \partial \cM, \text{ and } \ip{v}{-n(x)} \geq 0, \\
        P_\parallel(x, v) & \text{if } x \in \partial \cM, \text{ and } \ip{v}{-n(x)} \leq 0.
    \end{cases}
\end{equation*}
and define $P^+ : T^* \cM \rightarrow T^* \cM$ similarly. Note that in terms of $P^+$, Assumption \ref{assumption:splittingH} implies
\begin{equation*}
    \cL(x,v) = \inf_{(x,\hat{v}) \in (P^+)^{-1}(x,v)} \cL_0(x,\hat{v}).
\end{equation*}
Furthermore set $\Psi : T\cM \rightrightarrows T\cM$ the set-valued map:
\begin{equation} \label{eqn:set_valued_map_projection}
    \Psi(x,v) = \text{convex hull} \left(\{(x,v)\} \cup \{P^+(x,v)\} \right)
\end{equation}
and define $\Psi : T^*\cM \rightrightarrows T^* \cM$ similarly.

\begin{definition} \label{definition:boundaryHddagger}
Set $H_\ddagger$ with $\cD(H_\ddagger) = C_u^1(\cM)$ by $H_\ddagger f(x) = \cH_\ddagger (x, \dd f(x))$ with $\cH_\ddagger : T^* \cM \rightarrow \bR$ given by
\begin{equation*}
    \cH_\ddagger(x,p) = \min_{(x,\hat{p}) \in \Psi(x,p)} \cH_0(x,\hat{p}).
\end{equation*}
\end{definition}

\begin{remark} \label{remark:upperboundH_ddagger_example}
    Note that due to the orthogonal decomposition
    \begin{equation*}
    \cH_\ddagger(x,p) = \cH_\parallel(x,p) + \min_{(x,\hat{p}) \in \Psi(x,p)} \cH_\perp(x,\hat{p}).
    \end{equation*}
    and $\cH_\ddagger \leq \cH_0$ and $\cH_\ddagger(x,p) = \cH_0(x,p)$ for $x \notin \partial \cM$. 
\end{remark}

\begin{theorem} \label{theorem:boundary_existence}
Consider Setting \ref{setting:subspace} and let Assumption \ref{assumption:splittingH} be satisfied. Let $\cL, H_\dagger, H_\ddagger$ be as in Definitions \ref{definition:adaptedBoundaryL},  \ref{definition:boundaryHdagger} and \ref{definition:boundaryHddagger} respectively.

Then the outcomes of Theorem \ref{theorem:stationary} and \ref{theorem:timedep} hold.
\end{theorem}

\begin{remark}
    To establish that the pair of Hamiltonians $\cH_\dagger,\cH_\ddagger$ is sufficiently tight, we show in Section \ref{section:comparion_boundary} below in the simplified context that $\cM$ is geodesically convex a comparison principle can be established using standard estimates on $\cH_0$.
\end{remark}

\begin{proof}[Proof of Theorem \ref{theorem:boundary_existence}]
    The result follows by verification of Assumptions \ref{assumption:Lagrangian}, \ref{assumption:Lyapunov}, \ref{assumption:Hdagger} and \ref{assumption:Hddagger}.

    Assumption \ref{assumption:Lagrangian} \ref{item:assumption:Lagrangian:lowerbound} and \ref{item:assumption:Lagrangian:superlinear_growth} follow immediately from the same bounds for $\cL_0$ taken from Setting \ref{setting:subspace} \ref{setting:Lagrangian}.

    Assumption \ref{assumption:Lyapunov} follows from Setting \ref{setting:subspace} \ref{setting:Upsilon}, Proposition \ref{proposition:verifyLyapunov} and that $\cH_\dagger \leq 0 \vee \cH_0$.

    Assumption \ref{assumption:Hdagger} follows by Propositions \ref{proposition:boundaryH_dagger}, whereas Assumption \ref{assumption:Hddagger} follows from Proposition \ref{proposition:boundary_Hddagger_lsc} and Corollary \ref{corollary:recovery_curve_construction}. To conclude Assumption \ref{assumption:Lagrangian} \ref{item:assumption:Lagrangian:upperbound} also follows from  Corollary \ref{corollary:recovery_curve_construction}.
\end{proof}

\subsection{Analysis of Hamiltonians}

\begin{proposition} \label{proposition:boundaryH_dagger}
    $\cH_\dagger$ satisfies Assumption \ref{assumption:Hdagger}.
\end{proposition}

For the proof below, we denote $\cL^* : T^* \cM \rightarrow \bR$, defined by the fibre-wise Legendre transform: $\cL^*(x,p) := \sup_{v \in T \cM} \left\{\ip{p}{v} - \cL(x,v)\right\}$. Note that $\cL^*$ equals the \textit{reflected Hamiltonian} $H^\pi$ of \cite{Da06} and (53) in \cite{Li85boundary}.
    
\begin{proof}[Proof of Proposition \ref{proposition:boundaryH_dagger}]
    As $\cH_0$ is continuous, it follows by construction that $\cH_\dagger$ is upper semi-continuous.
    
    For establishing the Young upper bound, we use Proposition \ref{proposition:identify_multid_H} to identify $\cH_\dagger$ introduced in Definition \ref{definition:boundaryHdagger} in terms of a Legendre dual. In this appendix, we work for fixed $x$, and use the projection $P_1 := P_\perp$ and $P_2 := P_\parallel$. Similarly, we write $\cH_{1,0} = \cH_\perp, \cH_{2,0} = \cH_\parallel, \cL_\perp = \cL_{1,0}$ and $\cL_\parallel = \cL_{2,0}$.

    We thus obtain from Proposition \ref{proposition:identify_multid_H} that
    \begin{equation*}
        \cH_\dagger(x,p) = \begin{cases}
            \cH_0(x,p) & \text{if } x \notin \partial \cM, \\
            \cL^*(x,p) \vee \cH_0(x,p) & \text{if } x \in \partial \cM.
        \end{cases}
    \end{equation*}
    As $\cH_0 = \cL_0^* = \cL^*$ also on the interior of $\cM$, we find that 
    \begin{equation*}
        \ip{p}{v} \leq \cL^*(x,p) + \cL(x,v) \leq \cH_\dagger(x,p) + \cL(x,v).
    \end{equation*}
\end{proof}

We proceed with the analysis for $\cH_\ddagger$. We start out by establishing lower semi-continuity.

\begin{proposition} \label{proposition:boundary_Hddagger_lsc}
    $\cH_\ddagger$ is lower semi-continuous
\end{proposition}

    We will establish this result by two intermediate lemmas. Recall that the defitions for hemi-continuity are given in Appendix \ref{appendix:differential_inclusions}.

\begin{lemma} \label{lemma:Psi_uhc}
    The map $\Psi: T\cM \rightrightarrows T\cM$ is upper hemi-continuous and hemi-continuous when restricted to $\partial \cM$.
\end{lemma}

\begin{proof}
    Restricted to $x \in \partial \cM$ the map $P^+ : T^*\cM \rightarrow T^* \cM$ is continuous. This implies that the set-valued map $\Psi$ is hemi-continuous by Theorem 17.37 in \cite{AlBo06}.

    The map $\Psi$ on $\cM$ is thus the union of a continuous map on $\cM$ with a hemi-continuous map on the closed subspace $\partial \cM$ establishing the result by Theorem 17.27 in \cite{AlBo06}.
\end{proof}

\begin{lemma} \label{lemma:example_H_ddagger_contOnBoundary}
    The map $\cH_\ddagger(x,p)$ is continuous on $\partial \cM$.
\end{lemma}

\begin{proof}
    Restricted to $\partial \cM$ the map $\Psi$ is hemi-continuous by Lemma \ref{lemma:Psi_uhc}. Thus $\cH_\ddagger$ is continuous by the Berge maximum theorem, see Theorem 17.31 in \cite{AlBo06}.
\end{proof}
\begin{proof}[Proof of Proposition \ref{proposition:boundary_Hddagger_lsc}]
As in Remark \ref{remark:upperboundH_ddagger_example}, note that $\cH_\ddagger(x,p) \leq \cH_0(x,p)$ as $(x,p) \in \Psi(x,p)$. It follows that $\cH_\ddagger(x,p) = \min \left\{\cH_0(x,p), \cH_\ddagger(x,p) \right\}$. As $\cH_\ddagger$ is continuous on $\partial \cM$ by Lemma \ref{lemma:example_H_ddagger_contOnBoundary}, it follows that $\cH_\ddagger$ is the minimum of a continuous map $\cH_0$ on $\cM$ with the continuous map $\cH_\ddagger$ on the closed subspace $\partial \cM$, meaning it is lower semi-continuous.
\end{proof}

We proceed with the verification of Assumption \ref{assumption:Hddagger} \ref{item:assumptionHddagger:Young_lowerbound}, which involves solving a differential inclusion. As noted above, this will be based on Appendix \ref{appendix:differential_inclusions}, which will be used here in the form of a local trivialization of the (co)-tangent bundle.

In cases where all points are interior points, the result follows by solving $\dot{x} \in \partial_p \cH_0(x,\dd f(x))$. As we are not allowed to exit $\cM$, we need to project away any velocity pointing outward.

\begin{definition}
    Denote $\Phi : T^* \cM \rightrightarrows T \cM$ by
    \begin{equation*}
        \Phi := P^+ \circ \partial_p \cH_0.
    \end{equation*}
\end{definition}

To obtain our final result, we essentially need to establish that
\begin{itemize}
    \item for any $f \in \cD(H_\ddagger)$ and starting point $x_0 \in \cM$ the existence of curves $\gamma$ satisfying $\gamma(0) = x_0$ and  $\dot{\gamma} \in \Phi(\gamma(t), \dd f(\gamma(t))$.
    \item For any $v \in \Phi(x,p)$,
    \begin{equation*}
        \ip{p}{v} \geq \cH_\ddagger(x,p) + \cL(x,v)
    \end{equation*}
\end{itemize}
as these statements combined give the final result after minor additional work. We start with the latter.

\begin{proposition} \label{proposition:boundary_YoungRecovery}
    For $(x,p) \in T^*\cM$ and any $v \in \Phi(x,p)$ we have
    \begin{equation*}
       \ip{p}{v} \geq \cH_\ddagger(x,p) + \cL(x,v).
    \end{equation*}
\end{proposition}

\begin{proof}

    For $x \notin \partial \cM$ we obtain
    \begin{equation*}
        \ip{p}{v} = \cH_0(x,p) + \cL_0(x,v) = \cH_\ddagger(x,p) + \cL(x,v) 
    \end{equation*}
    by Theorem 23.5 in \cite{Ro70}. It thus suffices to consider $x \in \partial \cM$. First of all, consider $p$ and $v \in \Phi(x,p)$ satisfying $\ip{v}{-n(x)} > 0$. Then $\cL(x,v) = \cL_0(x,v)$ and $v \in \partial \cH_0(x,p)$. Thus, similarly as above
    \begin{equation*}
        \ip{p}{v} = \cH_0(x,p) + \cL_0(x,v)= \cH_\ddagger(x,p) + \cL(x,v).
    \end{equation*}
    The remaining case is $(x,p) \in T^*\cM$ with $x \in \partial \cM$ and $v \in \Phi(x,p)$ satisfying $\ip{v}{-n(x)} = 0$. We split the proof in two cases
    \begin{itemize}
        \item Case 1: $\cL(x,v) = \cL_0(x,v)$.
        \item Case 2: $\cL(x,v) = \cL_\parallel(x,v)$.
    \end{itemize}

\textbf{Case 1}: Suppose that $\cL(x,v) = \cL_0(x,v)$. Consider
\begin{equation*}
    \hat{p} \in \argmin \left\{\cH_0(x,\cdot) \, \middle| \, P_\parallel(x,\hat{p}) = P_\parallel(x,p)\right\}
\end{equation*}
Note that $\ip{\hat{p}}{v} = \ip{p}{v}$ and that this $\hat{p}$ exists because there is a $\hat{v}$ satisfying $\ip{\hat{v}}{-n(x)} \geq 0$ with $\cL(x,\hat{v}) = 0$ and the orthogonal decomposition of $\cH_0$. Thus
\begin{equation*}
    \ip{p}{v} = \ip{\hat{p}}{v} = \cH_0(x,\hat{p}) + \cL_0(x,v) \geq \cH_0(x,p) + \cL(x,v) \geq \cH_(x,p) + \cL(x,v)
\end{equation*}
as $\cL = \cL_0$, the choice of $\hat{p}$ and Remark \ref{remark:upperboundH_ddagger_example}.

\smallskip

\textbf{Case 2}: As a final case consider the setting in which $\cL(x,v) = \cL_\parallel(x,v)$. As $\ip{v}{-n(x)} = 0$, we have $v \in \partial_p \cH_\parallel(x,p)$, implying as above that
    \begin{equation} \label{eqn:Young_parallel}
        \ip{p}{v} = \cH_\parallel(x,p) + \cL_\parallel(x,v) = \cH_\parallel(x,p) + \cL(x,v)
    \end{equation}
    by Remark \ref{remark:upperboundH_ddagger_example}. As for all $\hat{v}$ with $\cL_0(x,v) = 0$ we have $\ip{\hat{v}}{-n(x)} < 0$, we have $\ip{p}{-n(x)} \leq 0$ and consequently $\cH_\parallel(x,p) \geq \cH_\ddagger(x,p)$. Combining this with \eqref{eqn:Young_parallel} establishes the claim.
\end{proof}

We proceed with the verification that $\Phi$ is an appropriate map for solving differential inclusions using Appendix \ref{appendix:differential_inclusions}.

\begin{lemma} \label{lemma:Phi_uhc_onboundary}
    $\Phi : T^* \cM \rightrightarrows T \cM$ is upper hemi-continuous on $\partial \cM$.
\end{lemma}

\begin{proof}
    $P^+$ is a continuous map on $\partial \cM$. A subgradient map $\partial_p \cH_0(x,p)$ of a continuous convex function is upper hemi-continuous by Theorem 24.4 in \cite{Ro70}. The result thus follows by Theorem 17.23 in \cite{AlBo06}.
\end{proof}

The next result replaces $\Phi$ constructed from $P^+$ and $\partial_p \cH_0$ by $F$ constructed by $\Psi$ and$\partial_p \cH_0$ in the definition of the set-valued map. This extension will allow us to apply Theorem \ref{theorem:solve_differential_inclusion}. We will additionally show that $F$ restricted to allowable directions yields back $\Phi$.

\begin{proposition}\label{proposition:flow_uhc}
    Set $F : T^*\cM \rightrightarrows T\cM$ by 
    \begin{equation*}
        F := \Psi \circ \partial_p \cH_0.
    \end{equation*}
    Then $F$ is an upper hemi-continuous map on $\cM$ with closed convex values. In addition, if $x \in \partial \cM$, we have
    \begin{equation*}
            \left\{(x,v) \in F(x,v) \, \middle| \, \ip{v}{-n(x)} \geq 0 \right\} = \Phi(x,p).
    \end{equation*}
\end{proposition}

Before proving the result, we give a corollary and an auxiliary lemma used in the proof. The following follows from Proposition \ref{proposition:flow_uhc} by Theorem 17.23 of \cite{AlBo06}.
\begin{corollary}\label{corollary:flow_uhc}
    Let $f \in C^1(\cM)$ and set $F_f$ by
    \begin{equation*}
        F_f(x) := F(x,\dd f(x)).
    \end{equation*}
    Then $F_f : \cM \rightarrow T \cM$ is upper hemi-continuous with $\ip{v}{-n(x)} \geq 0$ for any $x \in \partial \cM$ and $v \in T^*_x \cM$.
\end{corollary}

The proof of Proposition \ref{proposition:flow_uhc} is based on the following auxiliary lemma.

\begin{lemma} \label{lemma_Psi_absorbedByP+}
    For $(x,v)$ with $x \in \partial \cM$ we have
    \begin{equation*}
        \Psi(x,v) \cap \left\{(x,v) \in T\cM \, \middle| \, \ip{v}{-n(x)} \geq 0\right\} = P^+(x,v).
    \end{equation*}
\end{lemma}

\begin{proof}
    By construction, we have
    \begin{equation*}
        P^+(x,v) \subseteq \Psi(x,v) \cap \left\{(x,v) \in T\cM \, \middle| \, \ip{v}{-n(x)} \geq 0\right\}. 
    \end{equation*}
    For the other inclusion, it suffices to consider $(x,v)$ with $x \in \partial \cM$ and $P^+(x,v) \neq (x,v)$, or equivalently $\ip{v}{-n(x)} < 0$. This means, however that 
    \begin{equation*}
        \Psi(x,v) = \left\{P_\parallel(x,v) + \lambda P_\perp(x,v) \, \middle| \, \lambda \in [0,1] \right\}
    \end{equation*}
    Intersecting this with $\left\{(x,v) \in T\cM \, \middle| \, \ip{v}{-n(x)} \geq 0\right\}$ leaves only $\lambda =0$ establishing the claim.
\end{proof}

\begin{proof}[Proof of Proposition \ref{proposition:flow_uhc}]
    By Theorem 17.23 in \cite{AlBo06} $F$ is upper hemi-continuous as the composition of two upper hemi-continuous maps, cf. Lemma \ref{lemma:Psi_uhc} and Theorem 24.4 in \cite{Ro70}.

    The map is obtained as the convex hull of the upper hemi-continuous map $(x,p) \mapsto \partial_p \cH_0(x,p)$ together with the map $(x,p) \mapsto \Phi(x,p)$ which is upper hemi-continuous on the boundary by Lemma \ref{lemma:Phi_uhc_onboundary}. It has convex values by construction. The final equality follows by Lemma \ref{lemma_Psi_absorbedByP+}.    
\end{proof}

\begin{proposition} \label{proposition:boundary_recoveryCurve}
    For any $f \in C^1_u(\cM)$ with $\inf_x \cH_0(x, \dd f(x)) > - \infty$, and $x_0 \in \cM$, there is a absolutely continuous $\gamma$ in $\cM$ satisfying $\gamma(0) = x_0$ and for almost every $t$
    \begin{equation*}
        \ip{\dd f(\gamma(t))}{\dot{\gamma}(t)} \geq \cH_\ddagger(\gamma(t), \dd f(\gamma(t))) + \cL(\gamma(t),\dot{\gamma}(t)).
    \end{equation*}
\end{proposition}

\begin{proof}
We invoke Theorem \ref{theorem:solve_differential_inclusion} to solve the differential inclusion $\dot{\gamma}(t) \in F_f(\gamma(t))$ with starting point $\gamma(0) = x$. To be able to do so, note that $F_f$ has closed convex values and is upper hemi-continouous by Corollary \ref{corollary:flow_uhc}, and that in fact $\dot{\gamma}(t) \in \Phi(\gamma(t),\dd \gamma(t))$. A growth bound on $F_f$ is not needed, due to the a-priori control on the reach of $\gamma$ obtained from Lemma \ref{lemma:curves_compact}.

The result now follows from Proposition \ref{proposition:boundary_YoungRecovery}.
\end{proof}

The following is now immediate.

\begin{corollary} \label{corollary:recovery_curve_construction}
    We have the following two statements
    \begin{enumerate}
        \item For any $f \in \cD(\cH_\ddagger) = C_u^1(\cM)$, let $\gamma$ be as in Proposition \ref{proposition:boundary_recoveryCurve}, then
        \begin{equation*}
            f(\gamma(t)) - f(\gamma(0)) \geq \int_0^t \cH_\ddagger(\gamma(t),\dd f(\gamma(t))) + \cL(\gamma(t),\dot{\gamma}(t)) \, dt
        \end{equation*}
        \item For any bounded continuous $\phi : \bR^+ \rightarrow \bR^+$, and $\gamma$ obtained from  Proposition \ref{proposition:boundary_recoveryCurve} for $f = 0$, we have
        \begin{equation*}
            \int_0^t \cL(\gamma(t),\dot{\gamma}(t)) \phi(t) \, dt \leq 0.
        \end{equation*}
    \end{enumerate}

\end{corollary}

\subsection{A basic comparison principle} \label{section:comparion_boundary}

Finding a pair of Hamiltonians for which to establish existence is in general not sufficient to establish a solid foundation for further theory. The next result is not as general as we can imagine such a theorem to be possible, and leave such an extension to future work. The next result, however, does establish that in principle the pair of Hamiltonians $\cH_\dagger,\cH_\ddagger$ is appropriately chosen.

In this context we want to point out the classical work \cite{BaLi91}, where a comparison principle is established for first order Hamilton-Jacobi equations with an "inward pointing" dynamics, corresponding to our case that $\cL = \cL_0$. The result below allows for both "inward" and "outward" pointing dynamics.

\begin{theorem} \label{theorem:comparison_boundary}
    Consider Setting \ref{setting:subspace} and suppose that Assumptions \ref{assumption:splittingH} is satisfied. In addition, suppose that $\cM$ is geodesically convex and that for any compact set $K \subseteq \cM$ there is a modulus of continuity $\omega_K : \bR^+ \rightarrow \bR^+$ such that for any $x,y \in K$ with $d(x,y) < i_K$ with $i_K$ the injectivity radius on $K$, we have
    \begin{equation} \label{eqn:comparison:base_estimate}
        \cH_0\left(x, \dd_x \frac{a}{2}d^2(x,y)\right) - \cH_0\left(y, - \dd_y \frac{a}{2}d^2(x,y)\right) \leq \omega_K(a d^2(x,y) + d(x,y))
    \end{equation}
    Then we have that
    \begin{enumerate}[(a)]
        \item the comparison principle holds for subsolutions $u_1$ and supersolutions $u_2$ for the pair of equations
    \begin{equation*}
        f - \lambda \cH_\dagger^\Upsilon f \leq h_1, \qquad
        f - \lambda \cH_\ddagger^\Upsilon f \geq h_2,
    \end{equation*}
    respectively.
    \item 
    the comparison principle holds for subsolutions $u_1$ and supersolutions $u_2$ for the pair of equations
    \begin{equation*}
            \begin{cases}
            \partial_t f - \cH_\dagger^\Upsilon f \leq 0 & \text{if } t \geq 0, \\
            f(0,\cdot) = h_1 & \text{if } t = 0,
        \end{cases} \qquad
        \begin{cases}
            \partial_t f - \cH_\ddagger^\Upsilon f \geq 0 & \text{if } t \geq 0, \\
            f(0,\cdot) = h_2 & \text{if } t = 0,
        \end{cases}
    \end{equation*}
    respectively.
    \end{enumerate}
\end{theorem}

\begin{proof}[Sketch of proof]
As we have a containment function in the domain of $H_\dagger^\Upsilon$ and $H_\ddagger^\Upsilon$, the comparison principle reduces, using the usual techniques to establishing an appropriate upper bound on the difference
\begin{equation}\label{eqn:comparison:base_estimate_proof}
        \cH_\dagger\left(x, p_{x,y}\right) - \cH_\ddagger\left(y, \hat{p}_{x,y}\right)
    \end{equation}
    for $x,y \in K$ with $d(x,y) < i_K$ for any compact set $K$ and where
    \begin{equation*}
        p_{x,y} := \dd_x \frac{a}{2}d^2(x,y), \qquad \hat{p}_{x,y} := - \dd_y \frac{a}{2}d^2(x,y).
    \end{equation*}
    We make the following four observations
    \begin{enumerate}
        \item If $x \notin \partial \cM$, then $\cH_\dagger(x,p_{x,y}) = \cH_0(x,p_{x,y})$. If $y \notin \partial \cM$ then $\cH_\ddagger(y,\hat{p}_{x,y}) = \cH_0(y,\hat{p}_{x,y})$
        \item If $x \in \partial \cM$ and $y \notin \partial \cM$, due to convexity of $\cM$, we have $\ip{p_{x,y}}{-n(x)} \leq 0$, so that $\cH_\dagger(x,p_{x,y}) = \cH_0(x,p_{x,y})$.
        \item if $x \notin \partial \cM$ and $y \in \partial \cM$, due to convexity of $\cM$, we have $\ip{\hat{p}_{x,y}}{-n(y)} \geq 0$, so that $\cH_\dagger(y,\hat{p}_{x,y}) = \cH_0(y,\hat{p}_{x,y})$.
        \item if $x \in \partial \cM$ and $y \in \partial \cM$, due to convexity of $\cM$, we have $\ip{p_{x,y}}{-n(x)} = \ip{\hat{p}_{x,y}}{-n(y)} = 0$, so that $\cH_\dagger(x,p_{x,y}) = \cH_0(x,p_{x,y})$ and $\cH_\dagger(y,\hat{p}_{x,y}) = \cH_0(y,\hat{p}_{x,y})$.
    \end{enumerate}
    We thus see that in all cases \eqref{eqn:comparison:base_estimate_proof} reduces to
    \begin{equation*}
        \cH_\dagger\left(x, p_{x,y}\right) - \cH_\ddagger\left(y, \hat{p}_{x,y}\right) = \cH_0\left(x, p_{x,y}\right) - \cH_0\left(y, \hat{p}_{x,y}\right) \leq \omega_K(a d^2(x,y) + d(x,y))
    \end{equation*}
    by \eqref{eqn:comparison:base_estimate} establishing the claim.
    
\end{proof}

\section{Fillipov differential equations} \label{section:Fillipov}

In this section, we will consider the Hamilton-Jacobi equation corresponding to the non-controlled Fillipov \cite{Fi13} differential equation
\begin{equation} \label{eqn:Fillipov}
    \begin{cases}
        \dot{x}(t) = \phi(x(t)), \\
        x(0) = x_0,
    \end{cases}
\end{equation}
with discontinuous $\phi$ on $\bR^d$. As in previous setting, we can imagine building a much more sophisticated example using our theory, but we stick to the base-line case to exhibit versatility.

\begin{setting} \label{setting:Fillipov}
    $\phi$ is locally bounded, and there is a a continuously differentiable function $\Upsilon : \bR^d \rightarrow [0,\infty)$ with $\inf_x \Upsilon = 0$ with compact sublevelsets such that $\sup_x \ip{\nabla \Upsilon(x)}{\phi(x)} \leq c_\Upsilon < \infty$.
\end{setting}

The usual interpretation of solutions of \eqref{eqn:Fillipov} is as a differential inclusion, see Section \ref{appendix:differential_inclusions}, for the map $\Phi :\bR^d \rightrightarrows \bR^d$ defined by
\begin{equation*}
    \Phi(x) := \bigcap_{\delta > 0} \bigcap_{\mu(N) = 0} \overline{\text{convex hull}} \left\{\phi(B_\delta(x)\setminus N)\right\}.
\end{equation*}

Regularity of $\Phi$ is well known, see e.f. \cite[Proposition 1.1 and Theorem 1.1]{BaRo05}.

\begin{proposition} \label{proposition:uhc_fillipov}
    Consider setting \ref{setting:Fillipov}. Then the set-valued map $\Phi :\bR^d \rightrightarrows \bR^d$ has non-empty, convex compact values and is upper hemi-continuous.
\end{proposition}

We can thus solve the differential equation by Theorem F.7 as we can control the reach of solutions a-priori by using $\Upsilon$ and Lemma \ref{lemma:curves_compact}. We embed the solutions in a Hamilton-Jacobi equation.

\begin{definition}
    Set
    \begin{equation*}
        \cL(x,v) = \begin{cases}
            0 & \text{if } v \in F(x), \\
            \infty & \text{otherwise}.
        \end{cases}
    \end{equation*}
    and
    \begin{equation*}
        \cH_\dagger(x,p) = \sup_{v \in \Phi(x)} \ip{p}{v}, \qquad \cH_\ddagger(x,p) = \inf_{v \in \Phi(x)} \ip{p}{v}
    \end{equation*}
    and set $H_\dagger$ and $H_\ddagger$ as usual with domains $C_c^1(\bR^d)$.
\end{definition}

\begin{theorem} \label{theorem:Fillipov}
     Consider setting \ref{setting:Fillipov}, then the conditions for \ref{theorem:stationary} and \ref{theorem:timedep}  are satisfied and the value functions corresponding to $\cL$ gives viscosity sub- and supersolutions for the Hamilton-Jacobi equation in terms of $\cH_\dagger$ and $\cH_\ddagger$ respectively.
\end{theorem}

We refrain from establishing a comparison principle. We, however, point to \cite{Fj18} for a uniqueness statement for solutions to $\dot{x} \in \Phi(x)$ in terms of a Osgood \cite{Os1898} modulus of continuity. We furthermore point to \cite[Page 587]{CrIsLi87} for a comparison principle for a Hamilton-Jacobi with Osgood controllable drift. Finally, we point at the notion of dissipative set valued maps, that would lead to proper bounds when used in combination with $\cH_\dagger,\cH_\ddagger$. We leave a full analysis for future work.

\begin{proof}[Proof of Theorem \ref{theorem:Fillipov}]
    It suffices to verify Assumptions \ref{assumption:Lagrangian}, \ref{assumption:Lyapunov}, \ref{assumption:Hdagger}, and \ref{assumption:Hddagger}, 

    We start with the proof of Assumptions \ref{assumption:Hdagger}, \ref{assumption:Hddagger}. First of all, by the upper semi-continuity of $F$, $\cH_\dagger$ and $\cH_\ddagger$ are upper and lower semi-continuous by Lemma 17.30 in \cite{AlBo06}. For the upper bound
    \begin{equation*}
        \ip{p}{v} \leq \cH_\dagger (x,p) + \cL(x,v)
    \end{equation*}
    we only need to consider $v \in \Phi(x)$, but then clearly $\ip{p}{v} \leq \cH_\dagger(x,p)$. This establishes Assumption \ref{assumption:Hdagger}. Assumption \ref{assumption:Hddagger} follows similarly after applying Theorem \ref{theorem:solve_differential_inclusion} using Proposition \ref{proposition:uhc_fillipov} after noting that the reach of solutions can be controlled using $\Upsilon$. The final property $\inf_x \cH_\ddagger(x,\dd f(x)) > - \infty$ follows as $\cD(H_\ddagger) = C_c^1(\bR^d)$. This establishes Assumption \ref{assumption:Hddagger}.

    \smallskip

    Assumption \ref{assumption:Lagrangian} is immediate given its definition and the application of the constructed curve in Assumption \ref{assumption:Hddagger} for $f = 0$. Assumption \ref{assumption:Lyapunov} is obtained using Proposition \ref{proposition:verifyLyapunov} using Setting \ref{setting:Fillipov} and the continuous differentiability of $\Upsilon$.
\end{proof}

\section{The dynamic programming principle} \label{section:dpp}

We start with an important property of the two value functions, $R_{\lambda,h}$ and $\rest$ in \eqref{def:eq:value_function} and \eqref{def:eq:value_function_timedep}, namely the \textit{Dynamic Programming Principle}. The proof of the following results are standard (see, for example, \cite{BaCD97}). We include them for completeness. 

As noted in the Introduction, a secondary expression for the resolvent is available via the use of integration by parts. A variant of the Dynamic Programming Principle for this representation is included without proof in Lemma \ref{lemma:DPP_IBP}.

\begin{proposition}[Dynamic Programming Principle]
	\label{prop:DPP}
Let $h \in C_b(\cM)$ and $Adm \subseteq AC ([0,\infty),\cM)$ be an admissible set of curves as in Definition \ref{definition:curves}.
    
Consider $\res$ and $\rest$ defined as in \eqref{def:eq:value_function} and \eqref{def:eq:value_function_timedep} respectively. Then, we have the following
	\begin{enumerate}[(a)]
		\item \label{item:DPP_resolvent} For all $x \in \cM$, $\lambda >0$ and all $T>0$ 
		\begin{multline}
			\tag{DPP}
			\label{eq:DPP}
			\res(x) = \sup_{\gamma \in Adm, \gamma(0)=x} \int_0^T e^{-\lambda^{-1}t}  \left(\frac{h(\gamma(t))}{\lambda} -  \cL(\gamma(t), \dot{\gamma}(t)) \right) \, dt\\ + e^{-\lambda^{-1}T} \res(\gamma(T)).
		\end{multline}
		\item \label{item:DPP_time} For all $x\in \cM$, $\lambda \geq 0$ and $0< \tau \leq t$, 
		\begin{equation}
			\label{eq:dpp:timedep}
			\tag{DPPt}
			\rest(x,t) = \sup_{\gamma\in Adm, \gamma(0) = x} \left\{ \int_0^\tau - e^{-\kappa s} \cL(\gamma(s),\dot{\gamma}(s))\, ds + e^{- \kappa\tau}\rest(\gamma(\tau),t-\tau) \right\}.
		\end{equation}
	\end{enumerate}	
\end{proposition}

\begin{proof} 
	The proofs of the two properties are both based on integral change of variables that are possible by the definition of $Adm$ which involves piece-wise connectable curves.
	
	\textit{Proof of \ref{item:DPP_resolvent}.}
	We call $u_{T}(x)$ the right-hand side of \eqref{eq:DPP}. 
    
    \textbf{Proof of  $\res(x) \leq u_T(x)$.}
    If $u_T(x)=+ \infty$, there is nothing to prove. Assume that $u_T(x)< + \infty$. We start out with taking any $\gamma$ in the definition of $\res(x)$ and upper bound it in terms of $u_T(x)$. Thus $\gamma \in Adm$ with $\gamma(0) = x$, then, 
\begin{align}
		& \int_0^\infty \left(\frac{h(\gamma(t))}{\lambda} -  \cL(\gamma(t), \dot{\gamma}(t)) \right) e^{-\lambda^{-1}t} \, dt \\
		& \qquad  = \int_0^T \left(\frac{h(\gamma(t))}{\lambda} -  \cL(\gamma(t), \dot{\gamma}(t)) \right) e^{-\lambda^{-1}t} \,  dt \\
		& \hspace{3cm} + \int_T^\infty \left(\frac{h(\gamma(t))}{\lambda} -  \cL(\gamma(t), \dot{\gamma}(t)) \right) e^{-\lambda^{-1}t} \, dt \\
		& \qquad = \int_0^T \left(\frac{h(\gamma(t))}{\lambda} -  \cL(\gamma(t), \dot{\gamma}(t)) \right) e^{-\lambda^{-1}t} \, dt \\
		& \hspace{3cm} + \int_0^\infty \left(\frac{h(\gamma(t+T))}{\lambda} -  \cL(\gamma(t+T), \dot{\gamma}(t+T)) \right) e^{-\lambda^{-1}(t+T)} \, dt  \\
		& \qquad =\int_0^T \left(\frac{h(\gamma(t))}{\lambda} -  \cL(\gamma(t), \dot{\gamma}(t)) \right) e^{-\lambda^{-1}t} \, dt \\
		& \hspace{3cm} + e^{-\lambda^{-1} T} \int_0^\infty \left(\frac{h(\tilde{\gamma}(t))}{\lambda} -  \cL(\tilde{\gamma}(t), \dot{\tilde{\gamma}}(t)) \right)  e^{-\lambda^{-1}t} \, dt, 
	\end{align}
with $\tilde{\gamma}(t) = \gamma(t+T)\in Adm$ by Definition \ref{definition:curves}. Taking the supremum over $Adm$ gives $\res(x) \leq u_T(x)$.

    \textbf{Proof of  $\res(x) \geq u_T(x)$.}
    If $\res(x) = + \infty$ there is nothing to prove. Thus assume $\res (x)< + \infty$ and consider $\eps >0 $, $\gamma \in Adm$ with $\gamma(0)=x$ and $\tilde{\gamma} \in Adm$ such that $\tilde{\gamma}(0)=\gamma(T)$ and 
   
    \begin{equation}
		\res(\gamma(T)) \leq \int_0^\infty \left(\frac{h(\tilde{\gamma}(t))}{\lambda} -  \cL(\tilde{\gamma}(t), \dot{\tilde{\gamma}}(t))\right)  e^{-\lambda^{-1}t} \, dt + \eps.
	\end{equation}
    
	Define now 
	\begin{equation}
		\bar{\gamma}(t) =
		\begin{cases}
			\gamma(t) & \text{if $0\leq t \leq T$;} \\
			\tilde{\gamma}(t-T) & \text{if $T\leq t$.}
		\end{cases}
	\end{equation}
	By Definition \ref{definition:curves} $\bar{\gamma} \in Adm$ with $\bar{\gamma}(0) = x$, so that 
\begin{align}
		\res(x) &\geq \int_0^\infty \left(\frac{h(\bar{\gamma}(t))}{\lambda} - \cL(\bar{\gamma}(t), \dot{\bar{\gamma}}(t))\right)  e^{-\lambda^{-1}t} \, dt \\
		& = \int_0^T \left(\frac{h(\gamma(t))}{\lambda} -  \cL(\gamma(t), \dot{\gamma}(t)) \right)  e^{-\lambda^{-1}t} \, dt \\
		& \hspace{2cm} + \int_T^\infty \left(\frac{h(\tilde{\gamma}(t-T))}{\lambda} -  \cL(\tilde{\gamma}(t - T), \dot{\tilde{\gamma}}(t - T)) \right) e^{-\lambda^{-1}t} \, dt \\
		& = \int_0^T \left(\frac{h(\gamma(t))}{\lambda} - \cL(\gamma(t), \dot{\gamma}(t)) \right)  e^{-\lambda^{-1}t} \, dt \\
		& \hspace{2cm} + e^{-\lambda^{-1}T} \int_0^\infty \left(\frac{h(\tilde{\gamma}(t))}{\lambda} - \cL(\tilde{\gamma}(t), \dot{\tilde{\gamma}}(t)) \right) e^{-\lambda^{-1}t} \, dt\\
		& \geq \int_0^T \left(\frac{h(\gamma(t))}{\lambda} - \cL(\gamma(t), \dot{\gamma}(t)) \right)  e^{-\lambda^{-1}t} \, dt + e^{-\lambda^{-1} T } \res(\gamma(T)) - \eps.
	\end{align}

	Due to the arbitrariness of $\eps$, we obtain that $\res(x) \geq u_T(x)$ and this conclude the proof.
	
	\textit{Proof of \ref{item:DPP_time}.} We call $v_{\tau} (x,t)$ the right-hand side of \eqref{eq:dpp:timedep}.
    
    \textbf{Proof of $\rest(x,t) \leq v_{\tau}(x,t)$.} For $t=\tau$, \eqref{eq:dpp:timedep} is the definition of $\rest$ \eqref{def:eq:value_function_timedep}. Suppose $t>\tau$. If $v_{\tau}(x,t) = +\infty$, there is nothing to prove. Assume that $v_{\tau}(x,t) < + \infty$. We start out with taking any $\gamma$ in the definition of $\rest(x,t)$ and upper bound it in terms of $v_\tau(x,t)$. Thus $\gamma \in Adm$ with $\gamma(0)=x$, then, 
	\begin{align}
		& \int_0^t - e^{-\kappa s} \cL(\gamma(s),\dot{\gamma}(s))\, ds + e^{- \kappa t}u_0(\gamma(s)) \\ 
        & \qquad = \int_0^\tau  - e^{-\kappa s} \cL(\gamma(s),\dot{\gamma}(s))\, ds + \int_\tau^t  - e^{- \kappa s} \cL(\gamma(s),\dot{\gamma}(s))\, ds + e^{- \kappa t}u_0(\gamma(t)) \\
		& \qquad = \int_0^\tau  - e^{-\kappa s} \cL(\gamma(s),\dot{\gamma}(s))\, ds + \int_0^{t-\tau}  - e^{-\kappa s-\kappa \tau} \cL(\gamma(s+\tau),\dot{\gamma}(s+\tau))\, ds  \\
        & \qquad \qquad + e^{-\kappa\tau} e^{-\kappa(t-\tau)} u_0(\gamma(t)) \\
		& \qquad = \int_0^\tau - e^{- \kappa s} \cL(\gamma(s),\dot{\gamma}(s))\, ds \\
        & \qquad \qquad + e^{-\kappa \tau}\left(\int_0^{t-\tau} - e^{-\kappa s} \cL(\tilde{\gamma}(s),\dot{\tilde{\gamma}}(s))\, ds + e^{-\kappa (t-\tau)} u_0(\tilde{\gamma}(t-\tau)) \right),
	\end{align}
	with $\tilde{\gamma}(t) = \gamma(t+\tau)\in Adm$ by Definition \ref{definition:curves}. Taking the supremum over $Adm$ yields $\rest(x,t) \leq v_\tau (x,t)$. 
    
	\textbf{Proof of $\rest(x,t) \geq v_\tau(x,t)$.} If $\rest(x,t) = +\infty$ there is nothing to prove. Thus assume $\rest(x,t) <\infty$ and consider $\eps>0$, $\gamma \in Adm$ such that $\gamma(0)=x$ and $\tilde{\gamma} \in Adm$ such that $\tilde{\gamma}(0)=\gamma(\tau)$ and 
	\begin{equation}
		\rest(\gamma(\tau), t-\tau) \leq \int_0^{t-\tau} - e^{-\kappa s}\cL(\tilde{\gamma}(s),\dot{\tilde{\gamma}}(s) \, ds + e^{-\kappa(t-\tau)} u_0(\tilde{\gamma}(t-\tau)) + \eps.
	\end{equation}
	Define
	\begin{equation}
		\bar{\gamma} (s) =
		\begin{cases}
			\gamma(s) &\text{if $s\leq \tau$;}\\
			\tilde{\gamma}(s - \tau) &\text{if $t>\tau$.}
		\end{cases}
	\end{equation}
	Then, by Definition \ref{definition:curves} $\bar{\gamma} \in Adm$ and $\bar{\gamma}(0) = x$, so that 
	\begin{align}
		\rest(x,t) &\geq \int_0^t  - e^{-\kappa s}\cL(\bar{\gamma}(s),\dot{\bar{\gamma}}(s)) \, ds + e^{-\kappa t} u_0(\bar{\gamma}(t))\\
		& = \int_0^\tau - e^{-\kappa s} \cL(\gamma(s),\dot{\gamma}(s))\, ds + \int_\tau^t - e^{-\kappa s} \cL(\tilde{\gamma}(s-\tau),\dot{\tilde{\gamma}}(s-\tau))\, ds\\
  & \qquad \qquad + e^{-\kappa t}u_0(\tilde{\gamma}(t-\tau))\\
		&= \int_0^\tau - e^{-\kappa s}\cL(\gamma(s),\dot{\gamma}(s) \, ds + e^{-\kappa \tau}\int_0^{t-\tau} -  e^{-\kappa s}\cL(\tilde{\gamma}(s),\dot{\tilde{\gamma}}(s)) \, ds \\
  & \qquad \qquad + e^{- \kappa \tau}e^{- \kappa (t-\tau)}u_0(\tilde{\gamma}(t-\tau))\\
		& \geq \int_0^\tau  - e^{- \kappa s}\cL(\gamma(s),\dot{\gamma}(s)) \, ds + e^{- \kappa \tau}\rest(\gamma(\tau),t-\tau) - \eps.
	\end{align}
	Due to the arbitrariness of $\eps$ we obtain that $\rest(x,t) \geq v_\tau (x,t)$ concluding the proof.
\end{proof}

For the resolvent, we will also use a secondary version of the dynamic programming principle that is based on integration by parts, cf. Lemma \ref{lemma:curvesIBP} below. The proof is similar.

\begin{lemma} \label{lemma:DPP_IBP}
For all $h \in C_b(\cM)$, $x \in \cM$, $\lambda >0$ and all $T>0$ 
		\begin{multline}
			\tag{DPP-IBP}
			\label{eq:DPP_IBP}
			\res(x) = \sup_{\gamma \in Adm, \gamma(0)=x} \left\{ \int_0^T \lambda^{-1} e^{-\lambda^{-1}t}  \left(h(\gamma(t)) - \int_0^t \cL(\gamma(s), \dot{\gamma}(s)) \, ds \right) \, dt \right. \\ 
            \left. + e^{-\lambda^{-1}T} \left( \res(\gamma(T)) - \int_0^T \cL(\gamma(t),\dot{\gamma}(t))  \, dt \right) \right\}.
		\end{multline}
\end{lemma}

The proof of the next result is straightforward, using that $\cL$ is globally lower bounded.

\begin{lemma}\label{lemma:curvesIBP}
Let $h\in C_b(\cM)$, and let $\cL$ and $Adm$ be as in Assumption \ref{assumption:Lagrangian}.

    For any $\lambda > 0$ and $\gamma \in Adm$, we have integration by parts
    \begin{equation*}
        \int_0^\infty e^{-\lambda^{-1}t}\cL(\gamma(t),\dot{\gamma}(t))  \,  \dd t = \int_0^\infty \lambda^{-1} e^{-\lambda^{-1}t} \int_0^t \cL(\gamma(s),\dot{\gamma}(s)) \, \dd s \,  \dd t. 
    \end{equation*}
    
\end{lemma}

\section{Boundedness of the candidate solutions}\label{sec:boundedness_solutions}

We next establish boundedness of the candidate solutions on the basis of the bounds on $\cL$ introduced in Assumptions \ref{assumption:Lagrangian} \ref{item:assumption:Lagrangian:lowerbound} and \ref{item:assumption:Lagrangian:upperbound}.

To formulate these bounds cleanly for $\rest$ and various $\kappa \geq 0$, we introduce the following short-hand notation. Set $\phi_\kappa(t) := e^{-\kappa t}$ and 
\begin{equation*}
    \Phi_\kappa(t) := \begin{cases}
    \frac{1-e^{-\kappa t}}{\kappa} & \text{if } \kappa > 0, \\
    t & \text{if } \kappa = 0.
    \end{cases}
\end{equation*}
For a function $f : \cM \to \R$, we denote with $\ssup{f}$ and $\iinf{f}$ the supremum and infimum of $f$ over $\cM$, respectively.

\begin{lemma} \label{lemma:uniform_bound}
Let Assumptions \ref{assumption:Lagrangian} \ref{item:assumption:Lagrangian:lowerbound} and \ref{item:assumption:Lagrangian:upperbound} be satisfied. Let $h \in C_b(\cM)$. 
    \begin{enumerate}[(a)]
        \item \label{item:proposition:uniform_bound:semigroup} For any $t,\kappa \geq 0$, we have
        \begin{align*}
            \ssup{\rest} & \leq \phi_\kappa(t) \ssup{h} - \Phi_\kappa(t) \cL_{\min}, \\
            \iinf{\rest(\cdot,t)}  & \geq \phi_\kappa(t)\iinf{h} - \Phi_\kappa(t) \cL_{\max}. 
        \end{align*}  
        \item \label{item:proposition:uniform_bound:resolvent} 
        For any $\lambda > 0$, we have
        \begin{align*}
        \ssup{ R_{\lambda,h}}  & \leq \ssup{h} - \lambda \cL_{\min}, \\
        \iinf{R_{\lambda,h}} & \geq \iinf{h}  - \lambda \cL_{\max}.
        \end{align*}.
    \end{enumerate}
    
\end{lemma}

\begin{proof}
    First, note that Assumption \ref{assumption:Lagrangian} \ref{item:assumption:Lagrangian:lowerbound} implies for any $\gamma \in Adm$
\begin{equation*}
    h(\gamma(t)) e^{-\kappa t} - \int_0^t e^{-\kappa s} \cL(\gamma(s),\dot{\gamma}(s))  \, ds  \leq e^{-\kappa t}\ssup{h} - \cL_{\min} \int_0^t e^{-\kappa s} \, ds, 
\end{equation*}
establishing the first statement of \ref{item:proposition:uniform_bound:semigroup} by taking a supremum over $\gamma$. For the infimum, consider any $x \in \cM$ and the curve $\gamma$ started from $x$ as in Assumption \ref{assumption:Lagrangian} \ref{item:assumption:Lagrangian:upperbound}, it follows that
\begin{equation*}
    h(\gamma(t)) e^{-\kappa t} - \int_0^t e^{-\kappa s} \cL(\gamma(s),\dot{\gamma}(s))  \, ds  \geq e^{-\kappa t}\iinf{h} - \cL_{\max} \int_0^t e^{-\kappa s} \, ds. 
\end{equation*}
The result now follows by optimizing over $x$, establishing \ref{item:proposition:uniform_bound:semigroup}. The result for the resolvent follows similarly by using the above estimate for $\kappa = 0$, and afterwards using the representation  \eqref{eqn:IPB_in_resolvent} obtained by performing integration by parts, cf. Lemma  \ref{lemma:curvesIBP}.  
\end{proof}

\section{Control via the Lyapunov function} \label{section:consequences_Lyapunov}

In order to establish our main results, Theorems \ref{theorem:stationary} and \ref{theorem:timedep}, proved in  Sections \ref{section:subsolutions} and \ref{section:supersolutions}, it is essential to show that the almost-optimal curves in the dynamic programming principle remain in compact sets. This is the content of this section. The argument has two parts: first, Lemma \ref{lemma:curves_compact} shows via the Lyapunov function $\Upsilon$ that curves with bounded Lagrangian cost remain in a compact set. Second, Proposition \ref{proposition:obtaining_upper_bound_on_integralL} establishes that, on the basis of Lemma \ref{lemma:uniform_bound}, the relevant admissible curves indeed have such a cost bound.





\subsection{Compact containment using the Lyapunov function} \label{subsection:containmentL}

We start out by exploiting the existence of $\Upsilon$.

\begin{lemma}
	\label{lemma:curves_compact}
    Let Assumption \ref{assumption:Lyapunov} be satisfied. Let $T>0$ and $K_0$ a compact in $\cM$ and $M > 0$ a constant. Set $\bar{M} := \sup_{K_0} \Upsilon + C_\Upsilon T + M$ and consider the compact set
    \begin{equation*}
        K  = \{ x \in E \, : \, \Upsilon(x) \leq \bar{M} \}.
    \end{equation*}
    For any $\gamma \in \cA\cC([0,T],\cM)$ satisfying $\gamma(0) \in K_0$ and
	\begin{equation}
		\int_0^T \cL(\gamma(s),\dot{\gamma}(s))   \, ds < M,
	\end{equation}
    we have $\gamma(t) \in K$ for all $t\leq T$.
\end{lemma}

\begin{proof}
	By Assumption \ref{assumption:Lyapunov} 
	\begin{align}
		\Upsilon(\gamma(t))
		&\leq \Upsilon(\gamma(0)) + C_\Upsilon t+ \int_0^t \cL(\gamma(s),\dot{\gamma}(s))   \, ds \\
        &\leq \Upsilon(\gamma(0)) + C_\Upsilon T+ \int_0^T \cL(\gamma(s),\dot{\gamma}(s))   \, ds \\
		&\leq \sup_{K_0}\Upsilon + C_\Upsilon T + M = \bar{M}, 
	\end{align}
	establishing the claim. 
\end{proof}

\subsection{Bounding the Lagrangian cost} \label{subsection:boundingLagrangianCost}

\begin{proposition} \label{proposition:obtaining_upper_bound_on_integralL}
    Let Assumption \ref{assumption:Lyapunov} be satisfied and let $h \in C_b(\cM)$.

    \begin{enumerate}[(a)]
        \item Let $\gamma$ be a $q$ admissible curve for $\rest$ with $\kappa \geq 0$, then
        \begin{equation*}
            \int_0^t \cL(\gamma(s),\dot{\gamma}(s))  \, ds \leq C + qe^{\kappa T}
        \end{equation*}
        for a constant $C = C(\ssup{h},\kappa, \cL_{\min}, \cL_{\max})$.
        \item Let $\gamma$ be a $q$ admissible curve for $\res$ with $\lambda > 0$, then
        \begin{equation*}
            \int_0^t \cL(\gamma(s),\dot{\gamma}(s))  \, ds \leq C + qe^{\lambda^{-1} T}
        \end{equation*}
        for a constant $C = C(\ssup{h},\lambda, \cL_{\min}, \cL_{\max})$.
    \end{enumerate}
\end{proposition}

\begin{proof}
We start with the proof for $\rest$. Let $\gamma$ be $q$ admissible for $\rest(x,t)$. By Lemma \ref{lemma:uniform_bound}, we have that $\iinf{\rest} > - \infty$. It follows that 
\begin{align*}
    \iinf{\rest} - q & \leq \rest(x,t) - q  \\
    & \leq e^{-\kappa t} h(\gamma(t)) - \int_0^t e^{-\kappa s} \cL(\gamma(s),\dot{\gamma}(s)) \, ds.
\end{align*}
Rearrangement gives
\begin{equation} \label{eqn:Lagrangian_upperbound_oncurve1}
    \int_0^t e^{-\kappa s} \cL(\gamma(s),\dot{\gamma}(s)) \, ds \leq e^{-\kappa t} \ssup{h} - \iinf{\rest} + q. 
\end{equation}
If $\kappa = 0$, we are done. If $\kappa > 0$, then a direct computation on the left-hand side yields
\begin{equation}\label{eqn:Lagrangian_upperbound_oncurve2}
    \int_0^t e^{-\kappa s} \cL(\gamma(s),\dot{\gamma}(s)) \, ds \geq e^{-\kappa t}  \int_0^t  \cL(\gamma(s),\dot{\gamma}(s)) \, ds +  \cL_{\min} \int_0^t \left( e^{-\kappa s} - e^{-\kappa t} \right)  \, ds.
\end{equation}
Combining \eqref{eqn:Lagrangian_upperbound_oncurve1} with \eqref{eqn:Lagrangian_upperbound_oncurve2} yields
\begin{equation*}
    \int_0^t  \cL(\gamma(s),\dot{\gamma}(s)) \, ds \leq \ssup{h} - e^{\kappa t} \iinf{\rest} + e^{\kappa t} q - \cL_{\min} \frac{e^{\kappa t} - 1 - \kappa t}{\kappa}.
\end{equation*}
This establishes the result for $\rest$.

For $\res$ a similar argument works. Again, for any $x \in M$ and $q > 0$ we can find a curve $\gamma$ with $\gamma(0) = x$, such that by the dynamic programming principle and integration by parts, Lemma \ref{lemma:DPP_IBP}, we have
\begin{align*}
    \iinf{R_{\lambda,h}} -q & \leq R_{\lambda,h}(x) -q \\
    & \leq e^{-\lambda^{-1} T} \left( R_{\lambda,h}(\gamma(T))- \int_0^T \cL(\gamma(t),\dot{\gamma}(t))  \, dt \right) \\
    & \qquad + \int_0^T \lambda^{-1}  e^{-\lambda^{-1} t} \left(h(\gamma(t)) - \int_0^t \cL(\gamma(s),\dot{\gamma}(s)) \, ds \right) \, dt.
\end{align*}

Bounding $h$ and $\res$ using Lemma \ref{lemma:uniform_bound}, and integration by parts on the double integral over $\cL$ leads to
\begin{equation*}
    \int_0^T e^{-\lambda^{-1}t} \cL(\gamma(t),\dot{\gamma}(t))  \, dt \leq e^{-\lambda^{-1}T} \ssup{R_{\lambda,h}} + \left(1 - e^{-\lambda^{-1}T}\right) \ssup{h} - \iinf{R_{\lambda,h}} + q
\end{equation*}
As in \eqref{eqn:Lagrangian_upperbound_oncurve2}, we thus obtain
\begin{multline*}
    \int_0^T \cL(\gamma(t),\dot{\gamma}(t))  \, dt \leq  \ssup{R_{\lambda,h}} + \left(e^{\lambda^{-1}T} - 1 \right) \ssup{h} - e^{\lambda^{-1}T}\iinf{R_{\lambda,h}} \\
    + qe^{\lambda^{-1}T} -  \cL_{\min} \frac{e^{\lambda^{-1} T} - 1 - \lambda^{-1} T}{\lambda^{-1}}
\end{multline*}
establishing the claim.

\end{proof}

\section{Local behavior of curves} \label{section:tightness}

In this section we develop estimates that allow us to control the local behavior of admissible curves. The key idea is that the superlinear growth of the Lagrangian, Assumption \ref{assumption:Lagrangian} \ref{item:assumption:Lagrangian:superlinear_growth}, enforces equi-continuity once uniform bounds on the average cost is established. Then, in Proposition \ref{prop:boundness} we provide a sufficient condition for the required bound. 

\begin{proposition}\label{proposition:convergence_of_curves}
    Let Assumption \ref{assumption:Lagrangian} \ref{item:assumption:Lagrangian:superlinear_growth} be satisfied. Let $T>0$ and $K\subseteq \cM$ a compact set. Let $\gamma_n \in \cA\cC([0,T],\cM)$ be such that $\gamma_n(t) \in K$, for every $n$ and $t\leq T$. Let $T_n \in [0,T]$ such that $T_n \downarrow 0$. Let $x_n := \gamma_n(0)$ converge to $x_0 \in E$. If
	\begin{equation}\label{eq:lemma_conv:boundness}
		\sup_n \sup_{t\leq T_n} \frac{1}{t} \int_0^t \cL(\gamma_n(s), \dot{\gamma}_n(s)) \, ds < \infty,
	\end{equation} 
	then
	\begin{equation}
		\lim_n \gamma_n(t_n) = x_0
	\end{equation}
	for all $t_n$ vanishing sequence faster then $T_n$.
\end{proposition}

We prove the result below. The next result gives us an approach to establishing \eqref{eq:lemma_conv:boundness}. 
 

\begin{proposition}\label{prop:boundness}
    Let Assumption \ref{assumption:Lagrangian} \ref{item:assumption:Lagrangian:superlinear_growth} be satisfied. Let $T>0$, $K$ a compact set and $C_1, C_2 \geq 0$. Let $\gamma_n \in \cA\cC([0,T],\cM)$ be such that $\gamma_n(t) \in K$, for every $n$ and $t\leq T$. Let $T_n \in [0,T]$ such that $T_n \downarrow 0$. Moreover, let $f \in C^1(\cM)$ such that the following holds for every $n$ and $t\leq T_n$
	\begin{equation}\label{eq:prop_boundness}
		\int_0^{t} \cL(\gamma_n(s), \dot{\gamma}_n(s)) \, ds \leq C_1 \int_0^{t} \langle \dd f(\gamma_n(s)), \dot{\gamma}_n(s)\rangle  + C_2 t. 
	\end{equation}
Then, 
\begin{equation}
	\sup_n \sup_{t\leq T_n} \frac{1}{t} \int_0^t \cL(\gamma_n(s), \dot{\gamma}_n(s)) \, ds < \infty.
\end{equation}
\end{proposition}

The proof of both Proposition \ref{proposition:convergence_of_curves} and Proposition \ref{prop:boundness} is based on the following auxiliary lemma which captures the superlinearity of $\cL$ and whose proof is inspired by Lemma 10.21 in \cite{FK06}. 

\begin{lemma}\label{lemma:condition8.9}
Let Assumption \ref{assumption:Lagrangian} \ref{item:assumption:Lagrangian:superlinear_growth} be satisfied. For every $f \in C^1(\cM)$ and compact set $K\subseteq \cM$ there exists a right continuous, non decreasing function $\psi_{f,K} : [0,\infty) \to [0,\infty)$ such that
\begin{enumerate}[(a)]
    \item \label{lemma:item:condition8.9:sublinear} $\lim_{r\to\infty} \frac{\psi_{f,K}(r)}{r} = 0$;
    \item \label{lemma:item:condition8.9:psi_L} $|\langle \dd f(x), v \rangle| \leq \psi_{f,K} (\cL(x,v) \vee 0)$ for all $x \in K, v \in T_x \cM$.
    \item \label{lemma:item:condition8.9:controlonpsi} Let $C>0$. Then there is a $m_{f,K} \in (0,C^{-1})$ and $r_{f,K} > 0$ such that
    \begin{equation*}
        \psi_{f,K}(r) \leq \begin{cases}
        \psi_{f,K}(r_{f,K}) & \text{if } r < r_{f,K}, \\
        m_{f,K} r & \text{if } r \geq r_{f,K}.
        \end{cases}
    \end{equation*}
\end{enumerate}
\end{lemma}


\begin{proof}[Proof of Proposition \ref{proposition:convergence_of_curves}]
As the space $C^1(\cM)$ determines which sequences in $\cM$ converge, it suffices to establish for arbitrary $g \in C^1(\cM)$ that
	\begin{equation} \label{eqn:proof_convergence_viaC1}
		\limsup_n g(\gamma_n(t_n)) \leq g(x_0).
	\end{equation}
    Denote by $M$ the constant in \eqref{eq:lemma_conv:boundness}. By Lemma \ref{lemma:condition8.9} we have
\begin{align*}
    g(\gamma_n(t_n)) - g(x_0) & = \int_0^{t_n} \ip{\dd g(\gamma_n(s)}{\dot{\gamma}_n(s)} \, \dd s \\
    & \leq \int_0^{t_n} \psi_{g,K}(\cL(\gamma_n(s),\dot{\gamma}_n(s)) \vee 0) \, \dd s \\
    & \leq \int_0^{t_n} m_{g,K} \cL(\gamma_n(s),\dot{\gamma}_n(s)) + \psi_{g,K}(r_{g,K}) \, \dd s \\
    & \leq t_n \left( M  m_{g,K} + \psi_{g,K}(r_{g,K}) \right)
\end{align*}
    establishing \eqref{eqn:proof_convergence_viaC1}.
\end{proof}

\begin{proof}[Proof of Proposition \ref{prop:boundness}]
    Note that if $C_1 = 0$ there is nothing to prove. Thus, assume $C_1 > 0$. Using $\psi_{f,K}$ as in Lemma \ref{lemma:condition8.9} \ref{lemma:item:condition8.9:controlonpsi} with $C = C_1$, we obtain for any $n$ and $t \leq T_n$
	\begin{equation}\label{eq:bound_subsolution}
    \begin{aligned}
		\int_0^{t} \cL(\gamma_n(s),\dot{\gamma}_n(s))  \, ds & \leq C_1\int_0^{t} \langle \dd f(\gamma_n(s)), \dot{\gamma}_n(s)\rangle \, ds + C_2 t \\
		& \leq C_1 \int_0^{t} \psi_{f,K} (\cL(\gamma_n(s),\dot{\gamma}_n(s)) \vee 0) \, ds + C_2 t \\
        & \leq C_1 m_{f,K} \int_0^{t}  \cL(\gamma_n(s),\dot{\gamma}_n(s)) \, ds + t \left(C_1 \psi_{f,K}(r_{f,K}) + C_2\right).
	\end{aligned}
    \end{equation}
Rearrangement leads to
\begin{equation*}
    \int_0^{t}  \cL(\gamma_n(s),\dot{\gamma}_n(s)) \, ds \leq t \frac{C_1 \psi_{f,K}(r_{f,K}) + C_2}{1- C_1 m_{f,K}}
\end{equation*}
which establishes the claim.
\end{proof}

\begin{proof}[Proof of Lemma \ref{lemma:condition8.9}]
   By Assumption \ref{assumption:Lagrangian} \ref{item:assumption:Lagrangian:superlinear_growth},
\begin{equation}
    \lim_{N\to\infty} \inf_{x\in K} \inf_{|v| \geq N} \frac{\cL(x,v)}{1+|v|} = + \infty.
\end{equation}
Define 
\begin{equation}
\varphi(s) = s \inf_{x\in K} \inf_{|v|\geq s} \frac{\cL(x,v) \vee 0}{1+|v|}.
\end{equation}
Then $\varphi : [0,\infty) \to [0,\infty)$ satisfies $\varphi(0) = 0$, is strictly increasing, and $r^{-1}\varphi(r) \to \infty$ for $r \to \infty$ establishing \ref{lemma:item:condition8.9:sublinear}. For every $f\in C^1(\cM)$ and a compact set $K$ there exists a constant $C_{f,K} >0$ such that 
\begin{equation}
    |\langle \dd f(x), v \rangle| \leq C_{f,K} |v| \qquad \text{for all $x \in K$ and $v \in T_x \cM$.}
\end{equation}
Define 
\begin{equation}
    \psi_{f,K}(r) := C_{f,K} \varphi^{-1}(r),
\end{equation}
then $\psi_{f,K}: [0,\infty) \to [0,\infty)$ is such that $r^{-1}\psi_{f,K}(r) \to 0 $ for $r\to\infty$ and since 
\begin{align}
   \varphi\left(C_{f,K}^{-1} |\langle \dd f(x), v\rangle| \right)\leq  \varphi(|v|) \leq \cL(x,v) \vee 0,
\end{align}
we conclude the proof of \ref{lemma:item:condition8.9:psi_L} by applying $\psi_{f,K}$ on the outer ends on the inequality.

We proceed with the proof of \ref{lemma:item:condition8.9:controlonpsi}. The bound for $r \geq r_{f,K}$ in terms of the constant $m_{f,K}$ and cut-off point $r_{f,K}$ can be found from \ref{lemma:item:condition8.9:sublinear}. The bound for $r < r_{f,K}$ then follows as $\psi_{f,K}$ is non-decreasing.
\end{proof}

\section{Existence of subsolutions} \label{section:subsolutions}

In the section we prove the subsolution part of Theorems \ref{theorem:stationary} and \ref{theorem:timedep}.

\subsection{The exponentially discounted case}

\begin{proof}[Proof of subsolution part of Theorem \ref{theorem:stationary}]

    Let $f_\dagger^\eps$ as in Definition \ref{definiton:HdaggerHddagger}. As $\ssup{h} < \infty$, by Lemma \ref{lemma:uniform_bound} also $\ssup{\res} < \infty$. This implies that $\res - f_\dagger^\varepsilon$ has compact superlevel sets.
  
    We can then apply Proposition \ref{proposition:usc_regularization_optimizer} with $\phi= \res$ and $f = f_\dagger^\eps$, yielding a converging sequence $x_n \rightarrow x_0$ contained in some compact set $K_0$ satisfying
    \begin{gather}
        \res(x_n) - f_\dagger^\eps (x_n) \geq \sup(\res-f_\dagger^\eps) - \frac{1}{n^2}, \label{eq:almostsupremum} \\
        (\res)^*(x_0) - f_\dagger^\eps(x_0) = \sup ((\res)^* - f_\dagger^\eps), \label{eq:subsolution_startingpoint}
    \end{gather}
    and
\begin{equation}\label{eq:subsolution_almostsupremum_limit}
        \lim_n \res(x_n) = (\res)^*(x_0).
    \end{equation}
    It thus suffices to establish that
    \begin{equation}\label{eq:sub_inequality}
        (\res)^*(x_0) - \lambda g_\dagger^\eps (x_0) - h(x_0) \leq 0.
    \end{equation}
    Let $\gamma_n \in Adm$ such that $\gamma_n(0) = x_n$ and almost optimizing \eqref{eq:DPP} of Proposition \ref{prop:DPP}:

    \begin{multline}\label{eq:intermediate}
    \res(x_n) - e^{-\lambda^{-1}/n} \res\left(\gamma_n\left(\frac{1}{n}\right)\right) \\
    \leq \int_0^{1/n} \left[ \frac{h(\gamma_n(t))}{\lambda} - \cL(\gamma_n(t),\dot{\gamma}_n(t)) \right] e^{-\lambda^{-1}t} \, dt + \frac{1}{n^2}. 
    \end{multline}
By Proposition \ref{proposition:obtaining_upper_bound_on_integralL} and the lower bound $\cL_{\min}$ on $\cL$, we have
\begin{equation*}
    \sup_n \int_0^{1/n} \cL(\gamma_n(s),\dot{\gamma}_n(s)) \, ds < \infty.
\end{equation*}

As $\gamma_n(0) \in K_0$, it follows by Lemma \ref{lemma:curves_compact} that there exists a compact $K$ such that 
\begin{equation} \label{eqn:subsolution_path_in_compact}
    \gamma_n(t) \in K \qquad \text{ for all } t \leq \frac{1}{n}.
\end{equation}
Rearranging \eqref{eq:almostsupremum}, using the fundamental theorem of calculus, and Assumption \ref{assumption:Lyapunov}, we find
\begin{align} \label{eq:differentiable_testfun}
&\res\left(\gamma_n\left(\frac{1}{n}\right)\right) - \res(x_n) \leq f_\dagger^\eps\left(\gamma_n\left(\frac{1}{n}\right)\right) - f_\dagger^\eps(x_n) + \frac{1}{n^2} \\
& = (1-\eps)\left[f\left(\gamma_n\left(\frac{1}{n}\right)\right) - f(\gamma_n(x_n))\right] + \eps \left[\Upsilon\left(\gamma_n\left(\frac{1}{n}\right)\right) - \Upsilon(\gamma_n(x_n)) \right] + \frac{1}{n^2}\\
& \leq (1-\eps)\int_0^{1/n} \langle \dd f\left(\gamma_n(s)\right), \dot{\gamma}_n(s)\rangle \, ds + \eps \int_0^{1/n} \cL(\gamma_n(s),\dot{\gamma}_n(s)) \, ds + \varepsilon \frac{1}{n}C_\Upsilon + \frac{1}{n^2}.
\end{align}
Then, combining \eqref{eq:intermediate} and \eqref{eq:differentiable_testfun} leads to
\begin{align}\label{eq:dpp+optimum}
    & - (1-\eps)\int_0^{1/n} \ip{\dd f\left(\gamma_n(s)\right)}{ \dot{\gamma}_n(s) } \, ds - \eps\int_0^{1/n} \cL(\gamma_n(s),\dot{\gamma}_n(s)) \, ds \\ 
    & \qquad - \varepsilon \frac{1}{n}C_\Upsilon - \frac{1}{n^2} + \left( 1- e^{-\lambda^{-1}/n} \right) \res\left(\gamma_n\left(\frac{1}{n}\right)\right) \\
    & \leq \res(x_n) - e^{-\lambda^{-1}/n}\res\left(\gamma_n\left(\frac{1}{n}\right)\right) \\
    & \leq \int_0^{1/n} \left[ \lambda^{-1} h(\gamma_n(t)) - \cL(\gamma_n(t),\dot{\gamma}_n(t))  \right]  e^{-\lambda^{-1}t} \, dt + \frac{1}{n^2}.
\end{align}
Working on the outer terms of the inequalities, dividing by $\frac{1}{n}$, and rearranging yields
\begin{subequations}
\label{eq:limitequation:sub}
\begin{align}
    & 0\leq n\left(e^{-\lambda^{-1}/n} -1 \right) \res\left(\gamma_n\left(\frac{1}{n}\right)\right) \label{line:resolvent} \\
    & \qquad + n \int_0^{1/n} \lambda^{-1} e^{-\lambda^{-1}t} h\left(\gamma_n(t)\right) \, dt \label{line:h} \\
    & \qquad + n (1-\eps) \int_0^{1/n} \langle \dd f(\gamma_n(t)),\dot{\gamma}_n(t) \rangle - \cL(\gamma_n(t),\dot{\gamma}_n(t))  \, dt + \varepsilon C_\Upsilon \label{line:testfunction}\\
    & \qquad - n \int_0^{1/n} \left(e^{-\lambda^{-1}t} -1 \right) \cL (\gamma_n(t),\dot{\gamma}_n(t)) \, dt + O\left(\frac{1}{n}\right). \label{line:lagrangian}
\end{align}
\end{subequations}

The proof is completed by showing that the limsup for $n \to \infty$ over the right-hand side of \eqref{eq:limitequation:sub} yields the inequality \eqref{eq:sub_inequality}.

\smallskip

Key to the analysis of the above limiting behaviour is to first establish
\begin{gather}
	\sup_n \sup_{t\leq \frac{1}{n}} \frac{1}{t} \int_0^t \cL(\gamma_n(s), \dot{\gamma}_n(s)) \, ds < \infty, \label{eq:proof_subsolution_boundness} \\ 
    \gamma_n \left(\frac{1}{n}\right) \to x_0. \label{eqn:convergence_gamman_to_x0}
\end{gather}
where the second is a direct consequence of the first due to Proposition \ref{proposition:convergence_of_curves}. The proof of \eqref{eq:proof_subsolution_boundness} derives from Assumption \ref{assumption:Lagrangian} \ref{item:assumption:Lagrangian:superlinear_growth}.

\textbf{Proof of \eqref{eq:proof_subsolution_boundness}.}
We will establish \eqref{eq:proof_subsolution_boundness} using Proposition \ref{prop:boundness} with $T_n = \frac{1}{n}$. To do so, we will establish \eqref{eq:prop_boundness}.  Putting all terms in \eqref{eq:limitequation:sub} involving $\cL$ on the left-hand side, we obtain
\begin{multline}
    n \int_0^{1/n} \left(e^{-\lambda^{-1}t} - \eps\right)\cL(\gamma_n(t),\dot{\gamma}_n(t)) \, dt \leq n (1-\eps) \int_0^{1/n} \langle \dd f(\gamma_n(t)),\dot{\gamma}_n(t) \rangle \, dt + \varepsilon C_\Upsilon \\
    + n\left(e^{-\lambda^{-1}/n} -1 \right) \res\left(\gamma_n\left(\frac{1}{n}\right)\right) + n \int_0^{1/n} \lambda^{-1} e^{-\lambda^{-1}t} h\left(\gamma_n(t)\right) \, dt  + O\left(\frac{1}{n}\right)
\end{multline}

By Lemma \ref{lemma:uniform_bound}, we have $\iinf{\res} > - \infty$, so that by \eqref{eqn:subsolution_path_in_compact} we can bound $h$ and $R_{\lambda,h}$ to obtain the finite upper bound
\begin{multline}
    n \int_0^{1/n} \left(e^{-\lambda^{-1}t}-\varepsilon\right) \cL(\gamma_n(t),\dot{\gamma}_n(t)) \, dt \leq n (1-\eps) \int_0^{1/n} \langle \dd f(\gamma_n(t)),\dot{\gamma}_n(t) \rangle \, dt \\
    - n\left(e^{-\lambda^{-1}/n} -1 \right)  \inf_{x \in K} \res\left(x\right) + \lambda^{-1} \sup_{x \in K} h(x)  + \eps C_\Upsilon  + O\left(\frac{1}{n}\right).
\end{multline}
This establishes \eqref{eq:prop_boundness}. \eqref{eq:proof_subsolution_boundness} now follows by Proposition \ref{prop:boundness}.

\medskip

We proceed with the taking the limsup in $n$ over the separate terms in \eqref{eq:limitequation:sub}.

\noindent\textbf{Limsup of \eqref{line:resolvent}:}

As the pre-factor in \eqref{line:resolvent} tends to $-\lambda$, we cannot directly use that $\gamma_n(1/n) \rightarrow x_n$ in combination with the upper semi-continuous regularization of $\res$. Instead, we will argue using \eqref{eq:subsolution_almostsupremum_limit}. We set this up via \eqref{eq:intermediate}. Note first that
\begin{equation}\label{line:resolvent_final}
    \limsup_n n\left(e^{-\lambda^{-1}/n} -1 \right) \res\left(\gamma_n\left(\frac{1}{n}\right)\right)= - \lambda^{-1} \liminf_n e^{-\lambda^{-1}/n} \res\left(\gamma_n\left(\frac{1}{n}\right)\right).
\end{equation}
By \eqref{eq:intermediate} and \eqref{eq:subsolution_almostsupremum_limit}  we have
\begin{align}
    & \liminf_n e^{-\lambda^{-1}/n}  \res\left(\gamma_n\left(\frac{1}{n}\right)\right) \\
    & \qquad \geq \liminf_n \res(x_n) - \int_0^{1/n} \left[ \frac{h(\gamma_n(t))}{\lambda} - \cL(\gamma_n(t),\dot{\gamma}_n(t)) \right] e^{-\lambda^{-1}t} \, dt - \frac{1}{n^2} \\
    & \qquad = (\res)^*(x_0) - \limsup_n  \int_0^{1/n} \left[ \frac{h(\gamma_n(t))}{\lambda} - \cL(\gamma_n(t),\dot{\gamma}_n(t)) \right] e^{-\lambda^{-1}t} \, dt
\end{align}
As the integral terms vanish due to \eqref{eqn:subsolution_path_in_compact} and Assumption \ref{assumption:Lagrangian} \ref{item:assumption:Lagrangian:lowerbound}, we conclude that
\begin{equation*}
    \limsup_n n\left(e^{-\lambda^{-1}/n} -1 \right) \res\left(\gamma_n\left(\frac{1}{n}\right)\right)= - \lambda^{-1}(\res)^*(x_0).
\end{equation*}




\textbf{Limsup of \eqref{line:h}:}
By \eqref{eqn:subsolution_path_in_compact}, the convergence of $\gamma_n(1/n)$ to $x_0$, cf. \eqref{eqn:convergence_gamman_to_x0}, the upper semi-continuity of $h$ and and Fatou's lemma, we get that
\begin{equation}\label{line:h_final}
     \limsup_{n} n \int_0^{1/n} \lambda^{-1} e^{-\lambda^{-1}t}h\left(\gamma_n(t)\right) \, dt \leq \lambda^{-1} h(x_0).
\end{equation}

\noindent\textbf{Limsup of \eqref{line:testfunction}:} By Assumption \ref{assumption:Hddagger} \ref{item:assumptionHdagger:Young_upperbound} we get
\begin{align}\label{line:testufunction_final}
    & \limsup_n n \int_0^{1/n} (1-\eps)\left(\langle \dd f(\gamma_n(t)),\dot{\gamma}_n(t)\rangle - \cL(\gamma_n(t),\dot{\gamma}_n(t)) \right)\, dt + \varepsilon C_\Upsilon \\
    & \qquad \leq \limsup_n n \int_0^{1/n} (1-\eps) \cH_\dagger(\gamma_n(t), \dd f(\gamma_n(t)))\, dt  + \eps C_\Upsilon\\
    & \qquad \leq g_\dagger^\eps(x_0).
\end{align}
where in the last equality we used Assumption \ref{assumption:Hdagger} \ref{item:assumptionHdagger:usc}, \eqref{eqn:subsolution_path_in_compact}, the convergence of $\gamma_n(1/n) \to x_0$ cf. \eqref{eqn:convergence_gamman_to_x0}, and Fatou's lemma.

\medskip

\noindent\textbf{Limit of \eqref{line:lagrangian}:} By \eqref{eq:proof_subsolution_boundness},
\begin{equation}\label{line:lagrangian_final}
\lim_n n \int_0^{1/n} (e^{-\lambda^{-1}t} -1) \cL(\gamma_n(t), \dot{\gamma}_n(t)) \, dt  + O\left(\frac{1}{n}\right)= 0.
\end{equation}
Combining \eqref{line:resolvent_final}, \eqref{line:h_final}, \eqref{line:testufunction_final}, and \eqref{line:lagrangian_final} in \eqref{eq:limitequation:sub}, we obtain
\begin{equation}
   0\leq - \lambda^{-1} (\res)^*(x_0) + g_\dagger^\eps(x_0) + \lambda^{-1} h(x_0),
\end{equation}
which concludes the proof.
\end{proof}


\subsection{The time dependent case}

\begin{proof}[Proof of the subsolution part of Theorem \ref{theorem:timedep}] 
	The proof follows the same line as in Theorem \ref{theorem:stationary}. For completeness we give the main steps. \\

Let $f_\dagger^\eps$ as in Definition \ref{definiton:HdaggerHddagger}. Applying Proposition \ref{proposition:usc_regularization_optimizer} with $\phi = \rest$ and $f= f_\dagger^\eps$ and $h \in C^1([0,T])$, there exists a sequence $(x_n, t_n)$ in a compact set converging to a point $(x_0,t_0)$ and such that 
\begin{gather}
    \rest(x_n,t_n) - f_\dagger^\eps(x_n) - h(t_n) \geq \sup(\rest-f_\dagger^\eps - h) -\frac{1}{n}, \\
    (\rest)^*(x_0,t_0) - f_\dagger^\eps(x_0) - h(t_0) = \sup(\rest - f_\dagger^\eps - h), \label{eq:subsolution_startingpoint:timedep}
\end{gather}
and
\begin{equation}
    \lim_n \rest(x_n,t_n) = (\rest)^*(x_0,t_0).
\end{equation}
It thus suffices to establish that 
\begin{equation}\label{eq:subsolution_inequality_time}
\begin{cases}
   \partial_t h(t_0) + \kappa (\rest)^*(x_0,t_0) - g_\dagger^\eps(x_0) \leq 0 &\text{if $t_0>0$;}\\
   [\partial_t h(t_0) - g_\dagger^\eps(x_0)] \wedge [(\rest)^*(t_0,x_0) - u_0(x)] \leq 0 &\text{if $t_0 = 0$.}
\end{cases}
\end{equation}
Let $\gamma_n \in Adm$ be such that $\gamma_n (0) = x_n$ and almost optimizing \eqref{eq:dpp:timedep} of Proposition \ref{prop:DPP}:
\begin{equation}\label{eq:quasioptimum_dpp}
  \rest(x_n,t_n) \leq \int_0^{1/n} - e^{-\kappa s} \cL(\gamma_n(s),\dot{\gamma}_n(s))\, ds + e^{-\kappa/n} \rest(\gamma_n(1/n), t_n - 1/n) +\frac{1}{n^2}.  
\end{equation}
As $h \in C_b(\cM)$, it follows by Proposition \ref{proposition:obtaining_upper_bound_on_integralL} and the lower bound $\cL_{\min}$ on $\cL$, that

\begin{equation*}
    \sup_n  \int_0^{1/n} \cL(\gamma_n(s),\dot{\gamma}_n(s)) \, ds < \infty.
\end{equation*}

Combining this with $\gamma_n(0) \in K_0$, Lemma \ref{lemma:curves_compact} yields the existence of a compact set $K$ such that 
\begin{equation} \label{eqn:subsolution_path_in_compact2}
    \gamma_n(t) \in K \qquad \text{ for all } t \leq \frac{1}{n}.
\end{equation}
Rewriting \eqref{eq:subsolution_startingpoint:timedep}, and afterwards using the fundamental theorem of calculus and Assumption \ref{assumption:Lyapunov}, we find 
\begin{align}\label{eq:differentiable_testfun_time}
    & \rest(\gamma_n(1/n), t_n - 1/n) - \rest(x_n,t_n)\nonumber\\
     & \qquad \leq f_\dagger^\eps(\gamma_n(1/n)) - f_\dagger^\eps(x_n) + h(t_n - 1/n) - h(t_n) + \frac{1}{n^2}\nonumber \\
     & \qquad  = (1-\eps)\left[f\left(\gamma_n\left(\frac{1}{n}\right)\right) - f(\gamma_n(x_n))\right] + \eps \left[\Upsilon\left(\gamma_n\left(\frac{1}{n}\right)\right) - \Upsilon(\gamma_n(x_n)) \right] \nonumber \\
     & \qquad  \qquad + h(t_n - 1/n) - h(t_n) + \frac{1}{n^2}\nonumber\\
    &  \qquad \leq (1-\eps)\int_0^{1/n} \langle \dd f (\gamma_n(s)) , \dot{\gamma}_n(s) \rangle \, ds + \eps \int_0^{1/n} \cL(\gamma_n(s), \dot{\gamma}_n(s)) \, ds + \varepsilon \frac{1}{n} C_\Upsilon  \nonumber \\
    & \qquad  \qquad + h(t_n - 1/n) - h(t_n) + \frac{1}{n^2}.
\end{align}
Combining \eqref{eq:quasioptimum_dpp} and \eqref{eq:differentiable_testfun_time}, we obtain
\begin{multline}\label{eq:dpp+optimum_timedep}
    - (1-\eps)\int_0^{1/n} \langle \dd f\left(\gamma_n(s)\right), \dot{\gamma}_n(s)\rangle \, ds - \eps \int_0^{1/n} \cL(\gamma_n(s), \dot{\gamma}_n(s)) \, ds - \varepsilon \frac{1}{n} C_\Upsilon \\
    + h(t_n) - h(t_n - 1/n) - \frac{1}{n^2} 
    + (1 - e^{-\kappa/n}) \rest(\gamma_n(1/n), t_n - 1/n) \\
    \leq \int_0^{1/n} - e^{- \kappa s} \cL(\gamma_n(s), \dot{\gamma}_n(s))\, ds + \frac{1}{n^2}.
\end{multline}
Dividing by $\frac{1}{n}$ and rearranging yields
\begin{subequations}\label{eq:limitequation:sub:timedep}
\begin{align}
    0 & \leq n (e^{-\kappa /n} - 1) \rest(\gamma_n(1/n), t_n - 1/n) \label{line:sub_rest} \\
    & \qquad + n (h(t_n - 1/n) - h(t_n))\label{line:sub_test_t}\\
     & \qquad  + n (1-\eps)\int_0^{1/n} \langle \dd f\left(\gamma_n(s)\right), \dot{\gamma}_n(s)\rangle - \cL(\gamma_n(s),\dot{\gamma}_n(s) \, ds + \varepsilon C_\Upsilon \label{line:sub_testfunc_t}\\
    & \qquad  + n \int_0^{1/n} (1 - e^{-\kappa s}) \cL(\gamma_n(s),\dot{\gamma}_n(s)) \, ds + O\left(\frac{1}{n}\right)\label{line:sub_lagrangian_t}.
\end{align}
\end{subequations}
We obtain \eqref{eq:subsolution_inequality_time} by taking the limsup in the separate terms of \eqref{eq:limitequation:sub:timedep}. From this point onward, the proof follows that of the subsolution part of Theorem \ref{theorem:stationary}, with the straightforward modification to the limit  in \eqref{line:sub_test_t}.

\end{proof}

\section{Existence of supersolutions} \label{section:supersolutions}

In the section we prove the supersolution part of Theorems \ref{theorem:stationary} and \ref{theorem:timedep}.

\subsection{The exponentially discounted case}

\begin{proof}[Proof of supersolution part of Theorem \ref{theorem:stationary}]

    Let $f_\ddagger^\eps$ be as in Definition \ref{definiton:HdaggerHddagger}. By Lemma \ref{lemma:uniform_bound} and Proposition  \ref{proposition:usc_regularization_optimizer}, with $\phi = -\res$ and $f= - f_\ddagger^\eps$, there exists a converging sequence $x_n \rightarrow x_0$ in some compact set $K_0$ satisfying
    \begin{align}\label{eq:supersolution_startingpoint} 
    &\res(x_n) - f_\ddagger^\eps (x_n) \leq \inf_x(\res-f_\ddagger^\eps) + \frac{1}{n^2},\\
        &(\res)_*(x_0) - f_\ddagger^\eps(x_0) = \inf_x ((\res)_* - f_\ddagger^\eps),
    \end{align}
    and 
    \begin{equation} \label{eq:supersolution_starting_limit}
        \lim_n \res(x_n) = (\res)_*(x_0).
    \end{equation}
    It thus suffices to establish that
    \begin{equation}\label{eq:sup_inequality}
        (\res)_*(x_0) - \lambda g_\ddagger^\eps (x_0) - h(x_0) \geq 0.
    \end{equation}

By Assumption \ref{assumption:Hddagger} \ref{item:assumptionHddagger:Young_lowerbound}, there exist curves $\gamma_n \in Adm$ with $\gamma_n(0) = x_n$ satisfying
\begin{equation}\label{eq:youngequality_resolvent}
        \int_0^{1/n}  \cL(\gamma_n(t),\dot{\gamma}_n(t)) \, dt  = f(\gamma_n(1/n)) - f(\gamma(0))  - \int_0^{1/n}  \cH_\ddagger(\gamma_n(t), \dd f (\gamma_n(t))) \, dt 
\end{equation} 
As $\gamma_n(0) \in K_0$, $f \in C^1_u(\cM)$ and Assumption \ref{assumption:Hddagger} \ref{item:assumptionHddagger:lowerBound}, we obtain by Lemma \ref{lemma:curves_compact} and \eqref{eq:youngequality_resolvent} that there exists a compact $K$ such that
    \begin{equation} \label{eqn:supersolution_compactcontainment}
        \gamma_n(t) \in K \qquad \text{ for all } t \leq \frac{1}{n}.
    \end{equation}

    By \eqref{eq:supersolution_startingpoint}, and afterwards Assumption \ref{assumption:Lyapunov} 
    \begin{align}
        &\res(x_n) - \res\left(\gamma_n\left(\frac{1}{n}\right)\right) \leq f_\ddagger^\eps (x_n) - f_\ddagger^\eps \left(\gamma_n\left(\frac{1}{n}\right)\right) + \frac{1}{n^2} \label{eqn:supersolution_boundf} \\
        & \qquad = (1+\eps)\left[f(\gamma_n(x_n))- f\left(\gamma_n\left(\frac{1}{n}\right)\right)\right] - \eps \left[\Upsilon(\gamma_n(x_n)) - \Upsilon\left(\gamma_n\left(\frac{1}{n}\right)\right)\right] + \frac{1}{n^2}\\
        & \qquad \leq - (1+\eps) \int_0^{1/n} \langle \dd f\left(\gamma_n(s)\right), \dot{\gamma}_n(s)\rangle \, ds + \eps \int_0^{1/n} \cL(\gamma_n(s), \dot{\gamma}_n(s)) \, ds + \varepsilon \frac{1}{n} C_\Upsilon + \frac{1}{n^2}.
    \end{align}
Moreover, by \eqref{eq:DPP} of Proposition \ref{prop:DPP}:
    \begin{multline} \label{eqn:supersolution_boundR}
        \res(x_n) - e^{-\lambda^{-1}/n} \res\left(\gamma_n\left(\frac{1}{n}\right)\right) \geq \\
        \int_0^{1/n} \left[\frac{h(\gamma_n(t))}{\lambda} - \cL(\gamma_n(t),\dot{\gamma}_n (t))  \right] e^{-\lambda^{-1}t} \, dt.
   \end{multline}
Combining \eqref{eqn:supersolution_boundf} with \eqref{eqn:supersolution_boundR} yields
\begin{align}\label{eq:dpp+optimum_sup}
    & \int_0^{1/n} \left[\lambda^{-1} h(\gamma_n(t)) -  \cL(\gamma_n(t),\dot{\gamma}_n (t)) \right]  e^{-\lambda^{-1}t} \, dt \\
    & \qquad \leq \res(x_n) - \res\left(\gamma_n \left(\frac{1}{n}\right)\right) + (1 - e^{-\lambda^{-1}/n} )\res\left(\gamma_n\left(\frac{1}{n}\right)\right) \\
    & \qquad \leq- (1+\eps) \int_0^{1/n} \langle \dd f\left(\gamma_n(s)\right), \dot{\gamma}_n(s)\rangle \, ds + \eps \int_0^{1/n} \cL(\gamma_n(s), \dot{\gamma}_n(s) \dd s \\ 
    & \qquad \qquad + \varepsilon \frac{1}{n} C_\Upsilon + (1 - e^{-\lambda^{-1}/n} )\res\left(\gamma_n\left(\frac{1}{n}\right)\right) + \frac{1}{n^2}.
\end{align}
Dividing by $1/n$ and rearranging leads to
\begin{subequations}
\label{eq:limitequation:supersol}
\begin{align}
    0 & \leq - n\left(e^{-\lambda^{-1}/n} -1 \right) \res\left(\gamma_n\left(\frac{1}{n}\right)\right) \label{line:supersol_res} \\
    & \qquad - n \int_0^{1/n} \lambda^{-1} e^{-\lambda^{-1}t} h\left(\gamma_n(t)\right) \, dt \label{line:supersol_h} \\
    & \qquad  - n (1+\eps) \int_0^{1/n} \langle \dd f\left(\gamma_n(s)\right), \dot{\gamma}_n(s)\rangle - \cL(\gamma_n(t),\dot{\gamma}_n(t)) \, dt \label{line:supersol_testfunction} + \varepsilon C_\Upsilon \\
    & \qquad  + n \int_0^{1/n} \left(e^{-\lambda^{-1}t} -1 \right) \cL (\gamma_n(t),\dot{\gamma}_n(t)) \, dt + O\left(\frac{1}{n}\right). \label{line:supersol_lagrangian}
\end{align}
\end{subequations}
We show now that taking the limsup in \eqref{eq:limitequation:supersol} as $n \to \infty$ leads to inequality \eqref{eq:sup_inequality}. We analyse \eqref{line:supersol_res}, \eqref{line:supersol_h}, \eqref{line:supersol_testfunction}, and \eqref{line:supersol_lagrangian} separately. As in the subsolution case, a key step in the analysis is to establish first
\begin{gather} 
	\sup_n \sup_{t \leq \frac{1}{n}} \frac{1}{t} \int_0^t \cL(\gamma_n(s), \dot{\gamma}_n(s)) \, ds < \infty, \label{eq:proof_supersolution_boundness} \\
    \gamma_n \left(\frac{1}{n}\right) \to x_0.  \label{eqn:convergence_gamman_to_x0_supersol}
\end{gather}
where the second follows from the first using Proposition \ref{proposition:convergence_of_curves}. In contrast to the subsolution proof, \eqref{eq:proof_supersolution_boundness} in this setting directly follows from  Proposition \ref{prop:boundness}, $\gamma_n(t) \in K$ for all $t\leq 1/n$, cf. \eqref{eqn:supersolution_compactcontainment}, \eqref{eq:youngequality_resolvent}, and Assumption \ref{assumption:Hddagger} \ref{item:assumptionHddagger:lowerBound}.


\medskip

\noindent\textbf{Limsup of \eqref{line:supersol_res}:} 

Working towards the use of \eqref{eq:supersolution_starting_limit}, note that
\begin{equation*}
    \limsup_n - n\left(e^{-\lambda^{-1}/n} -1 \right) \res\left(\gamma_n\left(\frac{1}{n}\right)\right) = \lambda^{-1} \limsup_n e^{-\lambda^{-1}/n} \res\left(\gamma_n\left(\frac{1}{n}\right)\right)
\end{equation*}
and that by \eqref{eqn:supersolution_boundR}
\begin{align*}
    & \limsup_n e^{-\lambda^{-1}/n} \res\left(\gamma_n\left(\frac{1}{n}\right)\right) \\
    & \qquad \leq \limsup_n \left(\res(x_n) - \int_0^{1/n} \left[\frac{h(\gamma_n(t))}{\lambda} - \cL(\gamma_n(t),\dot{\gamma}_n (t))  \right] e^{-\lambda^{-1}t} \, dt\right).
\end{align*}
We thus find that 
\begin{equation} \label{line:supersol_res_final}
    \limsup_n - n\left(e^{-\lambda^{-1}/n} -1 \right) \res\left(\gamma_n\left(\frac{1}{n}\right)\right) 
    \leq \lambda^{-1} (\res)_*(x_0)
\end{equation}
by \eqref{eq:supersolution_starting_limit}, \eqref{eqn:supersolution_compactcontainment} and \eqref{eq:proof_supersolution_boundness}.


\noindent\textbf{Limsup of \eqref{line:supersol_h}:}
By \eqref{eqn:supersolution_compactcontainment}, the convergence of $\gamma_n(1/n)$ to $x_0$, cf. \eqref{eqn:convergence_gamman_to_x0_supersol}, the lower semi-continuity of $h$ and Fatou's lemma, the limsup over \eqref{line:supersol_h} is 
\begin{equation}\label{line:supersol_h_final}
    \limsup_{n\to\infty} - n \int_0^{1/n} \lambda^{-1} e^{-\lambda^{-1}/n} h\left(\gamma_n(t)\right) \, dt \leq - \lambda^{-1} h(x_0).
\end{equation}
\noindent\textbf{Limsup of  \eqref{line:supersol_testfunction}:}
Recall that $\gamma_n$ is constructed such that \ri{ \eqref{eq:youngequality_resolvent}} holds. Then,
\begin{multline}\label{eq:testfunction_supersolution_legendretransf_equality}
    - n \int_0^{1/n} (1+\eps) \left(\langle \dd f( \gamma_n(t)),\dot{\gamma}_n(t) \rangle - \cL(\gamma_n(t),\dot{\gamma}_n(t))\right) \, dt \\
     = - n \int_0^{1/n} (1+\eps) \cH_\ddagger(\gamma_n(t),\dd f(\gamma_n(t))) \, dt.
\end{multline}
This yields
\begin{align}\label{line:supersol_testfunction_final}
    & \limsup_n - n (1+\eps) \int_0^{1/n} \langle \dd f\left(\gamma_n(s)\right), \dot{\gamma}_n(s)\rangle - \cL(\gamma_n(t),\dot{\gamma}_n(t)) \, dt + \varepsilon C_\Upsilon \\
    & \qquad \leq \limsup_n -n \int_0^{1/n} (1+\eps) \cH_\ddagger(\gamma_n(t),\dd f(\gamma_n(t))) \, dt + \eps C_\Upsilon\\
    & \qquad \leq - g_\ddagger^\eps(x_0),
\end{align}
where in the last line we used Assumption \ref{assumption:Hddagger} \ref{item:assumptionHddagger:lsc}, \eqref{eqn:supersolution_compactcontainment}, \eqref{eqn:convergence_gamman_to_x0_supersol} and Fatou's lemma.
\noindent \textbf{Limsup of \eqref{line:supersol_lagrangian}:} By \eqref{eqn:supersolution_compactcontainment} and Assumption \ref{assumption:Lagrangian} \ref{item:assumption:Lagrangian:lowerbound} we have
\begin{equation}\label{line:supersol_lagrangian_final}
\limsup_n n \int_0^{1/n} (e^{-\lambda^{-1}t} -1) \cL(\gamma_n(t), \dot{\gamma}_n(t)) \, dt  + O\left(\frac{1}{n}\right) \leq 0.
\end{equation}
Taking the limsup for $n \to \infty$ in \eqref{eq:limitequation:supersol}, putting together \eqref{line:supersol_res_final}, \eqref{line:supersol_h_final}, \eqref{line:supersol_testfunction_final} and \eqref{line:supersol_lagrangian_final}, we obtain that
\begin{equation}
    0 \leq \lambda^{-1} (\res)_*(x_0) - g_\dagger^\eps(x_0) - \lambda^{-1} h(x_0),
\end{equation}
which concludes the proof.
\end{proof}

\subsection{The time dependent case}

\begin{proof}[Proof of the supersolution part of Theorem \ref{theorem:timedep}] 

As in the subsolution case, we point out the key steps.

   Let $f_\ddagger^\eps$ be as in Definition \ref{definiton:HdaggerHddagger}. Applying Proposition \ref{proposition:usc_regularization_optimizer} to $\phi = - \rest$, there exists a sequence $(x_n, t_n)$ converging to a point $(x_0,t_0)$ and such that 
\begin{gather}
    \rest(x_n,t_n) - f_\ddagger^\eps(x_n) - h(t_n) \leq \inf(\rest-f_\ddagger^\eps - h) + \frac{1}{n^2},  \label{eq:supersolution_startingpoint:timedep} \\
    (\rest)_*(x_0,t_0) - f_\ddagger^\eps(x_0) - h(t_0) = \inf(\rest - f_\ddagger^\eps - h),
\end{gather}
and
\begin{equation}
    \lim_n \rest(x_n,t_n) = (\rest)_*(x_0,t_0).
\end{equation}
It thus suffices to establish that 
\begin{equation}\label{eq:supersolution_property_time}
\begin{cases}
   \partial_t h(t_0) + \kappa (\rest)_*(x_0,t_0) - g_\ddagger^\eps(x_0) \geq 0 &\text{if $t_0>0$;}\\
   [\partial_t h(t_0) - g_\ddagger^\eps(x_0)] \vee [(\rest)_*(t_0,x_0) - u_0(x)] \geq 0 &\text{if $t_0 = 0$.}
\end{cases}
\end{equation}

By Assumption \ref{assumption:Hddagger} \ref{item:assumptionHddagger:Young_lowerbound}, there exist curves $\gamma_n \in Adm$ with $\gamma_n(0) = x_n$ satisfying
\begin{multline}\label{eq:youngequality_Res}
    \int_0^{1/n}  \cL(\gamma_n(t),\dot{\gamma}_n(t)) \, dt  = f(\gamma_n(1/n)) - f(\gamma(0))  \\
 - \int_0^{1/n}  \cH_\ddagger(\gamma_n(t), \dd f (\gamma_n(t))) \, dt 
\end{multline}
As $\gamma_n(0) \in K_0$, $f \in C^1_u(\cM)$ and Assumption \ref{item:assumptionHddagger:lowerBound}, we obtain by Lemma \ref{lemma:curves_compact} and \eqref{eq:youngequality_Res} that there exists a compact $K$ such that
    \begin{equation} \label{eqn:supersolution_compactcontainment2}
        \gamma_n(t) \in K \qquad \text{ for all } t \leq \frac{1}{n}.
    \end{equation}

As $\gamma_n(0) \in K_0$, we obtain by Lemma \ref{lemma:curves_compact} and \eqref{eq:youngequality_Res} that there exists a compact $K$ such that
    \begin{equation} 
        \gamma_n(t) \in K \qquad \text{ for all } t \leq \frac{1}{n}.
    \end{equation} 
By \eqref{eq:supersolution_startingpoint:timedep}, the fundamental theorem of calculus, and Assumption \ref{assumption:Lyapunov} 
\begin{align}\label{eq:supersolution_startingpoint_timedep}
    & \rest(x_n,t_n) - \rest\left(\gamma_n\left(\frac{1}{n}\right),t_n - 1/n\right) \\
    & \qquad  \leq f_\ddagger^\eps(x_n) - f_\ddagger^\eps\left(\gamma_n\left(\frac{1}{n}\right)\right) + h(t_n) - h(t_n - 1/n) + \frac{1}{n^2} \\
    & \qquad  = (1+\eps)\left[f(\gamma_n(x_n))- f\left(\gamma_n\left(\frac{1}{n}\right)\right)\right] - \eps \left[\Upsilon(\gamma_n(x_n)) - \Upsilon\left(\gamma_n\left(\frac{1}{n}\right)\right)\right] + \frac{1}{n^2}\\ 
    & \qquad \qquad + h(t_n) - h(t_n - 1/n) + \frac{1}{n^2} \\
    & \qquad \leq - (1+\eps) \int_0^{1/n} \langle \dd f\left(\gamma_n(s)\right), \dot{\gamma}_n(s)\rangle \, ds + \eps \int_0^{1/n} \cL(\gamma_n(s), \dot{\gamma}_n(s)) \, ds + \varepsilon \frac{1}{n}C_\Upsilon \\
    & \qquad \qquad + h(t_n) - h(t_n - 1/n) + \frac{1}{n^2}.
\end{align}
Moreover, by \eqref{eq:dpp:timedep} of Proposition \ref{prop:DPP}:
\begin{equation}\label{eq:supersolution_timedep:dpp}
  \rest(x_n,t_n) \geq \int_0^{1/n} - e^{-\kappa s} \cL(\gamma_n(s),\dot{\gamma}_n(s))\, ds + e^{-\kappa /n} \rest(\gamma_n(1/n), t_n - 1/n).  
\end{equation}
Then, combining \eqref{eq:supersolution_timedep:dpp} and \eqref{eq:supersolution_startingpoint_timedep}, we obtain
\begin{align}\label{eq:dpp+optimum_timedep_supersolution}
    & \int_0^{1/n} - e^{-\kappa s} \cL(\gamma_n(s),\dot{\gamma}_n(s))\, ds \\
    & \qquad  \leq \rest(x_n,t_n) - \rest(\gamma_n(1/n), t_n - 1/n) + (1- e^{-\kappa/n})\rest(\gamma_n(1/n), t_n - 1/n) \\
    & \qquad  \leq - (1+\eps) \int_0^{1/n} \langle \dd f\left(\gamma_n(s)\right), \dot{\gamma}_n(s)\rangle \, ds + \eps \int_0^{1/n} \cL(\gamma_n(s), \dot{\gamma}_n(s)) \, ds + \varepsilon \frac{1}{n}C_\Upsilon \\
    &\qquad  \qquad + h(t_n) - h(t_n - 1/n) + (1- e^{-\kappa/n})\rest(\gamma_n(1/n), t_n - 1/n) + \frac{1}{n^2}.
\end{align}
After dividing by $\frac{1}{n}$, we get 
\begin{equation}\label{eq:limitequation:super:timedep}
\begin{aligned}
    & 0 \leq - n (e^{-\kappa/n} - 1) \rest(\gamma_n(1/n), t_n - 1/n) \\
    & \qquad  - n (h(t_n - 1/n) - h(t_n)) \\
    & \qquad  - n (1+\eps) \int_0^{1/n} \langle \dd f\left(\gamma_n(s)\right), \dot{\gamma}_n(s)\rangle - \cL(\gamma_n(s),\dot{\gamma}_n(s) \, ds  + \varepsilon C_\Upsilon  \\
    & \qquad  - n \int_0^{1/n} (1 - e^{-\kappa s}) \cL(\gamma_n(s),\dot{\gamma}_n(s)) \, ds + O\left(\frac{1}{n}\right).
\end{aligned}
\end{equation}
We establish \eqref{eq:supersolution_property_time} by taking the limsup for $n \to \infty$ for the separate terms of \eqref{eq:limitequation:super:timedep}. 

From this point onward the proof is analogous to that of the supersolution part of Theorem \ref{theorem:stationary}.

\end{proof}

\appendix

\section{Viscosity solutions} \label{appendix:viscosity}

We give here the definitions of viscosity solutions for a stationary and a time-dependent Hamilton-Jacobi equation. For an explanatory text on the notion of viscosity solutions and fields of applications, we refer to~\cite{CIL92}.

    \begin{definition}[Viscosity solutions for the stationary equation] \label{definition:viscosity_solutions}
	Let $A_\dagger : \cD(A_\dagger) \subseteq C_l(\cM) \to C_b(\cM)$ be an operator with domain $\mathcal{D}(A_\dagger)$, $\lambda > 0$ and $h_\dagger \in C_b(\cM)$. Consider the Hamilton-Jacobi equation
	\begin{equation}
		f - \lambda A_\dagger f = h_\dagger. \label{eqn:differential_equation} 
	\end{equation}
	We say that $u$ is a \textit{(viscosity) subsolution} of equation \eqref{eqn:differential_equation} if $u$ is bounded from above, upper semi-continuous, and if for every $f \in \cD(A_\dagger)$ there exists $x_0 \in \cM$ such that
	\begin{gather*}
		u(x_0) - f(x_0)  = \sup_x u(x) - f(x), \\
		u(x_0) - \lambda A_\dagger f(x_0) - h_\dagger(x_0) \leq 0.
	\end{gather*}
 Let $A_\ddagger: \cD(A_\ddagger) \subseteq C_u (\cM) \to C_b (\cM)$ be an operator with domain $\mathcal{D}(A_\ddagger)$, $\lambda > 0$ and $h_\ddagger \in C_b(\cM)$. Consider the Hamilton-Jacobi equation
	\begin{equation}
		f - \lambda A_\ddagger f = h_\ddagger. \label{eqn:differential_equation-ddagger} 
	\end{equation}
	We say that $v$ is a \textit{(viscosity) supersolution} of equation \eqref{eqn:differential_equation-ddagger} if $v$ is bounded from below, lower semi-continuous, and if for every $f \in \cD(A_\ddagger)$ there exists $x_0 \in \cM$ such that
	\begin{gather*}
		v(x_0) - f(x_0)  = \inf_x v(x) - f(x), \\
		v(x_0) - \lambda A_\ddagger f(x_0) - h_\ddagger(x_0) \geq 0.
	\end{gather*}
	We say that $u$ is a \textit{(viscosity) solution} of the set of equations \eqref{eqn:differential_equation} and \eqref{eqn:differential_equation-ddagger} if it is both a subsolution of \eqref{eqn:differential_equation} and a supersolution of \eqref{eqn:differential_equation-ddagger}.
 \end{definition}

 \begin{definition}[Viscosity solutions for the time-dependent equation]
    Let $A_\dagger : \cD(A_\dagger) \subseteq C_l(\cM) \to C_b(\cM)$ be an operator with domain $\mathcal{D}(A_\dagger)$ and $\kappa \geq 0$. Consider the Hamilton-Jacobi equation with the initial value,
    \begin{gather}
        \begin{cases}
            \partial_t u(t,x) +\kappa u(t,x) - A_\dagger u(t,\cdot)(x)  = 0, & \text{if } t > 0, \\
            u(0,x) = u_0(x) & \text{if } t = 0.
        \end{cases} \label{eqn:HJ_def_subsolution}
    \end{gather}
    Let $T>0$, $f\in D(A_\dagger)$ and $g\in C^1([0,T])$ and let $F_\dagger(x,t): \cM\times [0,T] \to \R$ be the function 
    \begin{equation}
    F_\dagger (x,t) = 
    \begin{cases}
        \partial_t g(t) +\kappa u(x,t) - A_\dagger f(x) & \text{if $t>0$,} \\
        \left[\partial_t g(t) +\kappa u(x,t) - A_\dagger f(x) \right] \wedge \left[u(t,x)-u_0(x) \right] & \text{if $t=0$.}
    \end{cases}
    \end{equation}
       We say that $u$ is a \textit{(viscosity) subsolution} for \eqref{eqn:HJ_def_subsolution} if it is bounded from above, upper semi-continuous, and for any $T > 0$ any $ f \in D(A)$ and any $g\in C^1([0,T])$ there exists a pair $(t_0,x_0) \in [0,T] \times E$ such that
        \begin{align*}
        u(t_0,x_0) - f(x_0) -g(t_0) & = \sup_{t\in[0,T],x} u(t,x) - f(x) - g(t), \\
        F_\dagger (x_0,t_0) & \leq 0.
        \end{align*}

    Let $A_\ddagger : \cD(A_\ddagger) \subseteq C_u(E) \to C_b(E)$ be an operator with domain $\mathcal{D}(A_\ddagger)$ and $\kappa \geq 0$. Consider the Hamilton-Jacobi equation with the initial value,
    \begin{gather}
        \begin{cases}
            \partial_t u(t,x) +\kappa u(x,t) - A_\ddagger u(t,\cdot)(x)  = 0, & \text{if } t > 0, \\
            u(0,x) = u_0(x) & \text{if } t = 0.
        \end{cases} \label{eqn:HJ_def_supersolution} 
    \end{gather}
    Let $T>0$, $f\in D(A_\ddagger)$ and $g\in C^1([0,T])$ and let $F_\ddagger(x,t): E\times [0,T] \to \R$ be the function 
    \begin{equation}
    F_\ddagger (x,t) = 
    \begin{cases}
        \partial_t g(t) + \kappa u(x,t) - A_\ddagger f(x) & \text{if $t>0$,} \\
        \left[\partial_t g(t) + \kappa u(x,t) - A_\ddagger f(x) \right] \vee \left[u(t,x)-u_0(x) \right] & \text{if $t=0$.}
    \end{cases}
    \end{equation}
        
         We say that $v$ is a viscosity supersolution for \eqref{eqn:HJ_def_supersolution} if it is bounded from below, lower semi-continuous, and for any $T > 0$ any $f \in D(A_\ddagger)$ and $g\in C^1([0,T])$ there exists a pair $(t_0,x_0) \in [0,T] \times E$ such that
        \begin{align*}
        u(t_0,x_0) - f(x_0) -g(t_0) & = \inf_{t\in[0,T],x} u(t,x) - f(x) - g(t), \\
        F_\ddagger (x_0,t_0) & \geq 0.
        \end{align*}
        We say that $u$ is a \textit{(viscosity) solution} of the set of equations \eqref{eqn:HJ_def_subsolution} and \eqref{eqn:HJ_def_supersolution} if it is both a subsolution of \eqref{eqn:HJ_def_subsolution} and a supersolution of \eqref{eqn:HJ_def_supersolution}.
\end{definition}

\section{Properties of semi-continuous functions} \label{appendix:semi_continuity}

The following two propositions on the behaviour of semi-continuous regularizations will be used for $\res$ and $\rest$ respectively. Recall that $\phi^*$ is the upper semi-continuous regularization of $\phi$.

\begin{proposition} \label{proposition:usc_regularization_optimizer}
    Let $\cY$ be a sequential space. Consider $\phi : \cY \to \bR$ and suppose that $f : \cY \to \bR$ is lower semi-continuous. Suppose that for any $c \in \bR$ the set
    \begin{equation*}
    A_c := \left\{x \in \cY \, \middle| \, \phi(x) - f(x) \geq c \right\}
    \end{equation*}
    is relatively compact in $\cY$. Then there exists a converging sequence $x_n\to x_0$ such that the following properties hold.
   \begin{enumerate}[(a)]
       \item \label{item:subsolution:existence_quasi_optimizer}
       $\phi(x_n) - f (x_n) \geq \sup (\phi - f) - \frac{1}{n}$,
     \item \label{item:subsolution:limitpoint_optimizer}  $\phi(x_0) - f(x_0) = \sup (\phi - f) = \sup (\phi^* - f).$
       \item \label{item:subsolution:convergence} $\lim_n \phi(x_n) = \phi^*(x_0)$. 
   \end{enumerate}
\end{proposition}

\begin{proof}
For the proof of \ref{item:subsolution:existence_quasi_optimizer}, note that for every $n\geq 1$, there exists $x_n$ such that 
    \begin{equation}\label{eq:almostopt}
           \phi(x_n) - f (x_n) \geq \sup (\phi - f) - \frac{1}{n}.
       \end{equation}
       As the sequence $\{x_n\}_{n\geq 1}$ is contained in $A_{\sup_{\phi-f} - 1}$ there is a subsequence, without loss of generality also denoted by $x_n$, converging to some $x_0$.
       
       For \ref{item:subsolution:limitpoint_optimizer}, the the upper semi continuity of $\phi^* - f = (\phi-f)^*$ implies
\begin{multline}
    \sup \phi^* - f \geq \phi^*(x_0) - f(x_0) \geq \limsup_n (\phi^*(x_n) - f(x_n)) \\ \geq \liminf_n \phi(x_n) - f(x_n) =\sup (\phi^* - f).
\end{multline}
As the outer suprema are equal, all inequalities are equalities.

We proceed with the proof of \ref{item:subsolution:convergence}.
 First note that by \ref{item:subsolution:existence_quasi_optimizer} and \ref{item:subsolution:limitpoint_optimizer}, we have that 
 \begin{equation}\label{eq:limit_difference}
     \lim_n \phi(x_n) - f(x_n) = \phi^*(x_0) - f(x_0).
 \end{equation}
 As $\limsup_n \phi(x_n) \leq \limsup_n \phi^*(x_n) \leq \phi^*(x_0)$, it suffices to establish $\liminf_n \phi(x_n) \geq \phi^*(x_0)$.

 Suppose by contradiction that 
 $\liminf_n \phi(x_n) = \underline{\phi} < \phi^*(x_0)$. Consider a subsequence $x_{n_m}$ such that $\limsup_n \phi(x_{n_m})= \underline{\phi}$, then
\begin{align}
    \limsup_n \phi(x_{n_m}) - f (x_{n_m}) &\leq \limsup_n \phi(x_{n_m}) - \liminf_n f (x_{n_m})\\
    &\leq \underline{\phi} - f (x_0) < \phi^*(x_0) - f(x_0), 
\end{align}
contradicting \eqref{eq:limit_difference}. This concludes the proof of \ref{item:subsolution:convergence}.
\end{proof}

\section{Global bounds of the Lagrangian}\label{appendix:lagrangian}

\begin{lemma} \label{lemma:bounds_on_L}
Let Assumptions \ref{assumption:Hdagger} and \ref{assumption:Hddagger} be satisfied.
\begin{enumerate}[(a)]
    \item \label{lemma:item:bounds_on_L:upperbound} We have
    \begin{equation*}
        \inf_{v \in T_xM} \cL(x,v)  \geq - \cH_\dagger(x,0).
    \end{equation*}
    \item \label{lemma:item:bounds_on_L:boundsH} Suppose that $0 \in \cD(H_\ddagger)$. Then
    \begin{equation*}
        \cH_\ddagger(x,0) \leq \cH_\dagger(x,0)
    \end{equation*}
    \item \label{lemma:item:bounds_on_L:lowerbound} Suppose that $0 \in \cD(H_\ddagger)$ and that $x \mapsto \inf_{v \in T_xM} \cL(x,v)$ is lower semi-continuous. Then
    \begin{equation*}
        \inf_{v \in T_xM} \cL(x,v)  \leq - \cH_\ddagger(x,0)
    \end{equation*}
\end{enumerate}
\end{lemma}
The lemma yields the direct estimate
\begin{equation} \label{eqn:globalBoundsH}
    \inf_x \cH_\ddagger(x,0) \leq - \sup_{x \in \cM} \inf_{v \in T_x\cM} \cL(x,v) \leq - \inf_{(x,v) \in T\cM} \cL(x,v) \leq \sup_x \cH_\dagger(x,0)
\end{equation}
showing that global upper and lower bounds on $x \mapsto \inf_{v \in T_x\cM} \cL(x,v)$ can be obtained via $\cH_\dagger,\cH_\ddagger$.
\begin{proof}[Proof of Lemma \ref{lemma:bounds_on_L}]
    We start with the proof of \ref{lemma:item:bounds_on_L:upperbound}. By Assumption \ref{assumption:Hdagger} \ref{item:assumptionHdagger:Young_upperbound} for $p=0$, we have
    \begin{equation*}
        \cL(x,v) \geq - \cH_\dagger(x,0).
    \end{equation*}
    Optimizing over $v$ gives the result.

    We proceed with the proof of \ref{lemma:item:bounds_on_L:boundsH}. Consider the curve of Assumption \ref{assumption:Hddagger} started in $x$ for $0 \in \cD(H_\ddagger)$. Then for any $T$, we have
    \begin{equation} \label{eqn:boundsL_fatou_setup}
        \frac{1}{T} \int_0^T \cL(\gamma(t),\dot{\gamma}(t)) + \cH_\ddagger(\gamma(t),0) \, dt \leq 0.
    \end{equation}
    By \ref{lemma:item:bounds_on_L:upperbound}, we thus obtain
    \begin{equation*}
        \frac{1}{T} \int_0^T  \cH_\ddagger(\gamma(t),0) - \cH_\dagger(\gamma(t),0) \, dt \leq 0.
    \end{equation*}
    As the integrand is lower semi-continuous in $t$, the result follows by taking $T \downarrow 0$ and Fatou's lemma. \ref{lemma:item:bounds_on_L:lowerbound} follows similarly working from \eqref{eqn:boundsL_fatou_setup} using that the integrable lower bound for Fatou can be constructed using $-\cH_\dagger$.
\end{proof}

\section{Verification of Assumption \ref{assumption:Lyapunov}} \label{section:verificationLyapunov}

This section is dedicated to a result that verifies Assumption \ref{assumption:Lyapunov} arguing through a global upper bound on the operator $\cH_\dagger$.

\begin{proposition} \label{proposition:verifyLyapunov}
Suppose that Assumption \ref{assumption:Hdagger} holds. Suppose that $\Upsilon \in \cD(H_\dagger)$ is a containment function as in Definition \ref{def:lyapunovfunction}. Then Assumption \ref{assumption:Lyapunov} is satisfied with
    \begin{equation*}
        C_\Upsilon := \sup_x \cH_\dagger(x,\dd \Upsilon(x)) < \infty.
    \end{equation*}
\end{proposition}

\begin{proof}
 Let $\gamma$ be any admissible curve. By Assumption \ref{assumption:Hdagger} \ref{item:assumptionHdagger:Young_upperbound} and $\Upsilon \in \cD(H_\dagger)\subseteq C^1(M)$:
\begin{equation*}
    \Upsilon(\gamma(t)) - \Upsilon(\gamma(0)) \leq \int_0^t \cH_\dagger(\gamma(s),\dd \Upsilon(\gamma(s))) \, ds    + \int_0^t  \cL(\gamma(s),\dot{\gamma}(s)) )  \, ds
\end{equation*}
which implies the claim. 

\end{proof}

\section{Splitting of convex conjugates} \label{section:splitting_convex}

In this section, we will analyze point-wise the duality between mappings $\cL$ and $\cH$ and how this duality splits over an orthogonal decomposition of the tangent and cotangent spaces.

In the text below, we will fix a base-point and study the duality per fixed base-point.

For a map $\Phi : \bR^d \rightarrow \bR$, we write
\begin{equation*}
    \Phi^*(p) = \sup_v \ip{p}{v} - \Phi(v)
\end{equation*}
the convex conjugate.

\subsection{Orthogonal decomposition of Hamiltonians} \label{section:decomposeH}

\begin{assumption} \label{assumption:appendix_H}
Let $P_1,P_2 : \bR^d \rightarrow \bR^d$ be orthogonal projections such that $P_1 \perp P_2$ and $P_2 + P_2 = \bONE$.

Let $\cH_1, \cH_2 : \bR^d \rightarrow \bR$ be lower semi-continuous and convex and be such that that $\cH_i(p) = \cH_i(P_ip)$.

Set $\cH = \cH_1 + \cH_2$ and $\cL_i = \cH_i^*$, $\cL = \cH^*$.
\end{assumption}

The main result of this subsection is a decomposition for $\cL$ in terms of $\cL_1$ and $\cL_2$.

\begin{proposition} \label{proposition:orthogonal_decomposition_L}
    Let Assumption \ref{assumption:appendix_H} be satisfied, then
    \begin{equation*}
        \cL(v) = \cL_1(P_1 v) + \cL_2(P_2 v).
    \end{equation*}
\end{proposition}

We prove the result on the basis of two auxiliary lemmas.

\begin{lemma} \label{lemma:infty_at_mismatched_speed}
    We have $\cL_i(v) = \infty$ if $P_i v \neq v$.
\end{lemma}

\begin{proof}
We argue for $i = 1$. As $\ip{p}{v} = \ip{P_1p}{P_1v} + \ip{P_2p}{P_2v}$, we have
\begin{align*}
    \cL_1(v) & = \sup_p \left\{  \ip{P_1p}{P_1v} + \ip{P_2p}{P_2v} - \cH_1(p) \right\} \\
    & = \sup_p \left\{  \ip{P_1p}{P_1v} + \ip{P_2p}{P_2v} - \cH_1(P_1 p) \right\}
\end{align*}
As $\cH_1$ does not depend on $P_2 p$, it follows that if $P_2v \neq 0$, then $\cL_1(v) = \infty$.
\end{proof}

\begin{lemma}[Convex duality for the inf-convolution, Theorem 16.4 of \cite{Ro70}] \label{lemma:duality_infconvolution}
Let $f_1,f_2$ be convex lsc functions (not everywhere equal to $\infty$). Set 
\begin{equation*}
    \left(f_1 \,\Box\, f_2\right)(v) = \inf \left\{f_1(v_1) + f_2(v_2) \, \middle| \, v_1 + v_2 = v \right\}.
\end{equation*}
Then $(f_1 \, \Box \, f_2)^*(p) = f_1^*(p) + f_2^*(p)$.
\end{lemma}

\begin{proof}[Proof of Proposition \ref{proposition:orthogonal_decomposition_L}]
    By Assumption \ref{assumption:appendix_H} $H_i$ are convex and lower semi-continuous. It follows that $\cL_i^* = H_i$. Using Lemma \ref{lemma:duality_infconvolution} we thus obtain
    \begin{align*}
        \cL & = \left(H_1 + H_2\right)^* \\
        & = \left(\cL_1^* + \cL_2^*\right)^* \\
        & = \left(\cL_1 \, \Box \, \cL_2\right)^{**} \\
        & = \cL_1 \, \Box \, \cL_2.
    \end{align*}
    In other words:
    \begin{equation*}
        \cL(v) = \inf \left\{\cL_1(v_1) + \cL_2(v_2) \, \middle| \, v_1 + v_2 = v \right\}.
    \end{equation*}
    As $\cL_i(v_i) = \infty$ if $P_i v_i \neq v_i$ by Lemma \ref{lemma:infty_at_mismatched_speed}, it suffices to consider $v_i$ satisfying $P_i v_i = v_i$. This however implies that the only choice is given by $v_i = P_i v$ leading to
    \begin{equation*}
        \cL(v) = \cL_1(P_1 v) + \cL_2(P_2 v)
    \end{equation*}
    establishing the claim.
\end{proof}

\subsection{Orthogonal decomposition of Lagrangians}\label{section:decomposeL}

This subsection reverts the analysis of previous subsection by starting out with prototype Lagrangians. In contrast to previous section, we do not assume convexity.

\begin{assumption} \label{assumption:appendix_L}
Let $P_1,P_2 : \bR^d \rightarrow \bR^d$ be orthogonal projections such that $P_1 \perp P_2$ and $P_2 + P_2 = \bONE$.

Let $\cL_1, \cL_2 : \bR^d \rightarrow \bR$ be lower semi-continuous and such that $\cL_i(v) = \infty$ if $P_i v \neq v$.

Set $\cL = \cL_1 + \cL_2$, $\cH_i = \cL_i^*$ and $\cH = \cL^*$.
\end{assumption}

The main result of this subsection is a decomposition for $\cH$ in terms of $\cH_1$ and $\cH_2$.

\begin{proposition} \label{proposition:orthogonal_decomposition_H}
    Let Assumption \ref{assumption:appendix_L} be satisfied. Then we have $\cH = \cH_1 + \cH_2$.
\end{proposition}

\begin{proof}
    The result follows immediately from Lemma \ref{lemma:duality_infconvolution} and the definitions of $\cL_1,\cL_2,\cL$, as well as $\cH_1,\cH_2$ and $\cH$.
\end{proof}

We close of by completing the symmetry.

\begin{lemma}\label{lemma:reductionH}
    We have $H_i(p) = H_i(P_ip)$
\end{lemma}

\begin{proof}
    We argue for $i=1$. Using orthogonality, we have for any $p$ and $v$ that $\ip{p}{v} = \ip{P_1p}{P_1 v} + \ip{P_2p}{P_2v}$. Thus
    \begin{equation*}
        \cH_1(p) = \sup_v \left\{\ip{P_1p}{P_1v} + \ip{P_2p}{P_2v} - \cL_1(v)  \right\}.
    \end{equation*}
    As $\cL_1(v) = \infty$ if $P_2v \neq 0$, the supremum can be restricted to $v$ such that $P_2v = 0$. But then value of $P_2 p$ is irrelevant. In other words $\cH_1(p) = \cH_1(P_1p)$.
\end{proof}

\subsection{Removal of a half-space: the one-dimensional case}

In this section, we consider $\cL_0 : \bR \rightarrow [0,\infty)$ convex and lower semi-continuous, satisfying $\inf_v \cL_0(v) = 0$.

We will next discuss the setting where we disallow negative speeds.

\begin{definition} \label{definition:1d_restrictedL}
\begin{enumerate}[(a)]
    \item\label{definition:item:1d_restrictedL_inward} Suppose there is some $\hat{v} \geq 0$ such that $\cL_0(\hat{v}) = 0$. Then set
    \begin{equation*}
        \cL(v) := \begin{cases}
            \cL_0(v) & \text{if } v \geq 0, \\
            \infty & \text{if } v < 0.
        \end{cases}
    \end{equation*}
    \item \label{definition:item:1d_restrictedL_outward} Suppose for all $\hat{v}$ such that $\cL_0(\hat{v}) = 0$ we have $\hat{v} < 0$. Then set
    \begin{equation*}
        \cL(v) := \begin{cases}
            \cL_0(v) & \text{if } v > 0, \\
            0 & \text{if } v = 0, \\
            \infty & \text{if } v < 0.
        \end{cases}
    \end{equation*}
\end{enumerate}
    Define $\cH_0 = \cL_0^*$ and $\cH = \cL^*$.
\end{definition}

We next identify $\cH$ in terms of $\cH_0$. We start with an auxiliary lemma that bounds $\cL^*$ from below by $\cH_0$ for momenta that correspond to positive speeds.

\begin{proposition}\label{proposition:identify_1d_H}
    Denote $\argmin \cH_0 := \inf \left\{ p \, \middle| \, \cH_0(p) > \inf \cH_0 \right\} \in [-\infty,\infty]$. Then
    \begin{equation*}
        \cH(p) = \begin{cases}
            \inf \cH_0(p) & \text{if } \argmin \cH_0 \leq 0  \text{ and } p \leq \argmin \cH_0, \\
            \cH_0(p) & \text{if } \argmin \cH_0 \leq 0, \text{ and } p \geq \argmin \cH_0, \\
            0 & \text{if } \argmin \cH_0 \geq 0 \text{ and } p \leq 0, \\
            0 \vee \cH_0(p) & \text{if } \argmin \cH_0 \geq 0, \text{ and } p \geq 0.
        \end{cases}
    \end{equation*}
\end{proposition}

\begin{remark}\label{remark:consistency_adaptedH}
    Note that for the case $\argmin \cH_0 = 0$, we find that $\inf \cH_0 = 0$ and $\cH_0(p) \geq 0$ for $p \geq 0$, establishing consistency between the two branches obtained in the proposition.
\end{remark}

We establish the proposition by working first on a slightly adapted Lagrangian.

\begin{lemma} \label{lemma:lowerboundL_intermsof_H0_adapted}
    Assume that there is some $v \geq 0$ satisfying $\cL_0(v) < \infty$. Write 
    \begin{equation*}
        \widehat{\cL}(v) := \begin{cases}
            \cL_0(v) & \text{if } v \geq 0, \\
            \infty & \text{if } v < 0,
        \end{cases}
    \end{equation*}
    and $\widehat{\cH} = \widehat{\cL}^*$. Denote $\argmin \cH_0 := \inf \left\{ p \, \middle| \, \cH_0(p) \geq \inf \cH_0 \right\} \in [-\infty,\infty)$, possibly $- \infty$. Then
    \begin{equation*}
        \widehat{H}(p) = \begin{cases}
            \inf \cH_0 & \text{if } p \leq \argmin \cH_0, \\
            \cH_0(p) & \text{ if } p \geq \argmin \cH_0.
        \end{cases}
    \end{equation*}
\end{lemma}

\begin{proof}
    Note that as $\widehat{\cL}(v) = \infty$ for $v < 0$, the map $\widehat{\cH}$ is non-decreasing. Furthermore, as $\widehat{\cL} \geq \cL_0$, it follows that $\widehat{\cH} \leq \cH_0$. If $p$ is such that there exists $v(p) \in \partial_p \cH_0(p)$ satisfying $v(p) \geq 0$, then by Theorem 23.5 in \cite{Ro70} 
    \begin{equation*}
        \ip{p}{v(p)} = \cL_0(v(p)) + \cH_0(p)
    \end{equation*}
    so that
    \begin{equation*}
        \widehat{\cH}(p) \geq \ip{p}{v(p)} - \widehat{\cL}(v(p)) = \ip{p}{v(p)} - \cL_0(v(p))  = \cH_0(p).
    \end{equation*}
    The result thus follows.
\end{proof}

\begin{proof}[Proof of Proposition \ref{proposition:identify_1d_H}]

We consider first the setting where $\argmin \cH_0 \leq 0$. It follows that there is some $\hat{v} \in \partial \cH_0(0)$ satisfying $\hat{v} \geq 0$ and $\cL_0(\hat{v}) = 0$. We are thus in the context of Definition \ref{definition:1d_restrictedL} \ref{definition:item:1d_restrictedL_inward}. In other words $\cL = \widehat{\cL}$ as in Lemma \ref{lemma:lowerboundL_intermsof_H0_adapted} and we are done.

\smallskip

We proceed with the case where $\argmin \cH_0 > 0$ implying that for any $\hat{v} \in \partial \cH_0(0)$, or equivalently $\cL_0(\hat{v}) = 0$, we have $\hat{v} < 0$. This places us in branch \ref{definition:item:1d_restrictedL_outward} of Definition \ref{definition:1d_restrictedL} with $\cL(0) = 0$ and $\widehat{\cL}(0) > 0$.

\smallskip

We down-specify to the case that $\argmin \cH_0 = \infty$. This implies that there cannot be any $v > 0$ satisfying $\cL_0(v) < \infty$. It follows that $\cL(v) = 0$ if $v = 0$ and $\infty$ otherwise. This establishes that $\cH(p) = 0$. Note that in this case $\cH_0(p) < 0$ for all $p > 0$ giving correctness of the result.

We proceed with the case $\argmin \cH_0 \in (0,\infty)$. There is thus a $\hat{v} \geq \argmin \cH_0$ satisfying $\cL_0(\hat{v}) < \infty$. Combined with our starting assumption, $0$ lies in the interior of the effective domain of the lower semi-continuous and convex function $\cL_0$. Consequently $\cL_0$ is continuous at $0$ and $\widehat{\cL}$ is the lower semi-continuous regularization of the restriction of $\cL_0$ to $(0,\infty)$. It follows by Theorem 12.2 in \cite{Ro70} that
    \begin{equation*}
        \sup_{v > 0} \ip{p}{v} - \cL(v) = \sup_{v > 0} \ip{p}{v} - \cL_0(v) = \widehat{\cH}(p).
    \end{equation*}
    Consequently
    \begin{equation}\label{eqn:upperboundH_via0H0}
        \cH(p) = \sup_{v \geq 0} \left\{\ip{p}{v} - \cL(v) \right\} = 0 \vee \sup_{v > 0} \left\{\ip{p}{v} - \cL(v) \right\} = 0 \vee \widehat{\cH}(p).
    \end{equation}
The result in this setting thus follows by reading of $\widehat{\cH}$ from Lemma \ref{lemma:lowerboundL_intermsof_H0_adapted}.
\end{proof}

For the upper semi-continuity in the spatial variable, cf. Assumption \ref{assumption:Hdagger} in the context of Section \ref{section:Riemmannian_manifold_withBoundary}, we will end up working with the maximum of $\cH_0$ with $\cH$. This can now easiliy be read of from Proposition \ref{proposition:identify_1d_H_max}.

\begin{proposition}\label{proposition:identify_1d_H_max}
    Consider the context of Proposition \ref{proposition:identify_1d_H}. Then
    \begin{equation*}
        \cH(p) \vee \cH_0(p) = \begin{cases}
            \cH_0(p) & \text{if } p \leq 0, \\
            0 \vee \cH_0(p) & \text{if } p \geq 0.
        \end{cases}
    \end{equation*}
\end{proposition}

\subsection{Removal of a half-space: the multi-dimensional case}

In the section, we combine the analysis of orthogonally splitting Lagrangians and Hamiltonians together with the one-dimensional analysis of restricting a Lagrangian.

Let $\cL_0 : \bR^d \rightarrow [0,\infty)$ be lower semi-continuous and convex with $\inf \cL_0 = 0$. Let $\cH_{0} = \cL_0^*$. Suppose that $v_0 \in \bR^d \setminus \{0\}$ and set $P_1 v = \ip{v}{v_0} v$ and $P_2 v = v - P_1 v$.

Set $\cH_{i,0}(p) := \cH_0(P_i p) $and $\cL_{i,0} = \cH_{i,0}^*$

\begin{assumption} \label{assumption:multi-d_splitting}
    We have that
    \begin{equation*}
        \cH_0(p) = \cH_{1,0}(p) + \cH_{2,0}(p) 
    \end{equation*}
    or equivalently
    \begin{equation*}
        \cL_0(v) = \cL_{1,0}(P_1 v) + \cL_{2,0}(P_2 v).
    \end{equation*}
\end{assumption}

\begin{definition}
   \begin{enumerate}[(a)]
    \item Suppose there is some $\hat{v}$ such that $\ip{\hat{v}}{v_0} \geq 0$ such that $\cL_0(\hat{v}) = 0$. Then set
    \begin{equation*}
        \cL(v) := \begin{cases}
            \cL_0(v) & \text{if } \ip{v}{v_0} \geq 0, \\
            \infty & \text{if } \ip{v}{v_0} < 0.
        \end{cases}
    \end{equation*}
    \item Suppose for all $\hat{v}$ such that $\cL_0(\hat{v}) = 0$ we have $\ip{\hat{v}}{v_0} < 0$. Then set
    \begin{equation*}
        \cL(v) := \begin{cases}
            \cL_0(v) & \text{if } \ip{v}{v_0} > 0, \\
            0 & \text{if } \ip{v}{v_0} = 0, \\
            \infty & \text{if } \ip{v}{v_0} < 0.
        \end{cases}
    \end{equation*}
\end{enumerate}
    Define $\cH = \cL^*$.
\end{definition}

As $P_1$ maps to a one dimensional subspace of $\bR^d$, we can equip it with the canonical total order. We will use this in the proposition below.

\begin{proposition}\label{proposition:identify_multid_H}
    We have
    \begin{equation*}
        \cH(p) \vee \cH_0(p) = \cH_{2,0}(p) + \begin{cases}
            \cH_{1,0}(p) & \text{if } P_1 p \leq 0, \\
            0 \vee \cH_{1,0}   & \text{if } P_1 p \geq 0.
        \end{cases}
    \end{equation*}
\end{proposition}

\begin{proof}
    The result follows by combining Propositions \ref{proposition:orthogonal_decomposition_H} and \ref{proposition:identify_1d_H} and \ref{proposition:identify_1d_H_max}.
\end{proof}

	\section{Hemi-continuity and Differential inclusions} \label{appendix:differential_inclusions}

    The following appendix is based on \cite{De92,Ku00,AlBo06}. We start out by treating hemi-continuity, after which we treat differential inclusions on the basis of upper hemi-continuous maps.
	
	\smallskip

    \subsection{Hemi continuous mappings}

    Solution theory of differential inclusions is approached via regularity of set-valued mappings. A good reference on set-valued mappings is \cite{AlBo06}. Let $\cX,\cY$ be two metric spaces.
    
	\begin{definition} \label{definition:hemi_continuity}
		Let $F : \cX \rightrightarrows \cY$ be a set-valued map with $F(x) \neq \emptyset$ for all $x \in \cX$.
		\begin{enumerate}[(i)]
			\item We say that $F$ is \textit{closed, compact or convex valued} if each set $F(x)$, $x \in \cX$ is closed, compact or convex, respectively.
			\item We say that $F$ is \textit{upper hemi-continuous at $x \in \cX$} if for each neighbourhood $\cU$ of $F(x)$ in $\cY$, there is a neighbourhood $\cV$ of $x$ in $\cX$ such that $F(\cV) \subseteq \cU$. 
			We say that $F$ is \textit{upper hemi-continuous} if it is upper hemi-continuous at every point. $F$ is upper hemi-continuous if and only if for each sequence $x_n \rightarrow x$ in $\cX$ and $\xi_n \in F(x_n)$ such that $\xi_n \rightarrow \xi$ we have $\xi \in F(x)$.
            \item We say that $F$ is \textit{lower hemi-continuous at $x \in \cX$} if for each open set $\cU$ such that $\cU \cap F(x) \neq \emptyset$ there is an open neighbourhood $\cV$ of $x$ such that $\cU \cap F(y) \neq \emptyset$ for all $y \in \cV$.
			We say that $F$ is \textit{lower hemi-continuous} if it is lower hemi-continuous at every point. $F$ is upper hemi-continuous if and only if for each sequence $x_n \rightarrow x$ in $\cX$ and $\xi \in F(x)$ there exists a subsequence $x_{n_k}$ and $\xi_k \in F(x_{n_k})$ such that $\xi_k \rightarrow \xi$.
            \item $F$ is called hemi-continuous if it is both upper and lower hemi-continuous.
		\end{enumerate}
	\end{definition}

    We have the following immediate consequence of the notion of upper hemi-continuity, cf. Theorem 17.27 in \cite{AlBo06}.
    
    \begin{lemma} \label{lemma:convexhull_UHC}
        Let $\cX_0$ be a metric space, and let $\cX_1,\dots,\cX_k$ be closed sets inside $\cX_0$. Let $F_0 : \cX_0 \rightrightarrows \cY$ be upper hemi-continuous with non-empty closed convex values and let $F_i :\cX_i \rightrightarrows \cY$ with closed convex values. Set
        \begin{equation*}
            F(x) := \text{convex hull} \bigcup_{i : x \in \cX_i} F_i(x).
        \end{equation*}
        Then $F$ is upper hemi-continuous with closed convex values.
    \end{lemma}

    In Section \ref{section:Riemmannian_manifold_withBoundary} we use compositions of set-valued maps.
    \begin{definition}
        For metric spacex $\cX,\cY,\cZ$ and set-valued mappings $F: \cX \rightrightarrows \cY$ and $G : \cY \rightrightarrows \cZ$, define $(G \circ F) : \cX \rightrightarrows \cZ$
    \begin{equation*}
        (G \circ F)(x) := \bigcup_{y \in F(x)} G(y).
    \end{equation*}
    \end{definition}
    
    \subsection{Differential inclusions}

    We next focus on differential inclusions on subsets of $\bR^d$ following \cite{De92,Ku00}.
	
	\begin{definition}
		Let $I \subseteq \bR$ be an interval with $0 \in I$, $D\subseteq \bR^d$, $x \in D$ and $F : D \rightrightarrows \bR^d$ as set valued map. A function $\gamma$ such that
		\begin{enumerate}[(a)]
			\item $\gamma : I \rightarrow D$ is absolutely continuous,
			\item $\gamma(0) = x$,
			\item $\dot{\gamma}(t) \in F(\gamma(t))$ for almost every $t \in I$
		\end{enumerate}
		is  called a \textit{solution of the differential inclusion} $\dot{\gamma} \in F(\gamma)$ a.e., $\gamma(0) = x$.
	\end{definition}

	If we assume sufficient regularity on $F$, we can ensure the existence of a solution to differential inclusions that remain inside $D$. We first introduce how set-valued maps interact with the notion of boundaries.

    \subsection{The tangent cone}

    The results by \cite{De92} to solve differential inclusions on closed sets, introduce the convenient notion of the \textit{Bouligand cotingent cone}. Although this definition is more general than needed in our rather well-behaved context, we introduce it for completeness and later reference. In Lemma \ref{lemma:identification_tangent_cone}, we identify the tangent cone in terms of the normal vector $n$, which allows us to apply the result on a local chart on a Riemannian manifold.

	\begin{definition}
		Let $D \subseteq \bR^d$ be a closed non-empty set. The tangent cone to $D$ at $x$ is
		\begin{equation*}
		T_D(x) := \left\{v \in \bR^d \, \middle| \, \liminf_{t \downarrow 0} \frac{d(x + t v, D)}{t} = 0\right\}.
		\end{equation*}
	\end{definition}

    Note that $T_D$ can be interpreted as set-valued map.
    
    \begin{lemma} \label{lemma:identification_tangent_cone}
    Let $g : \bR^d \rightarrow \bR^d$ be $C^1$ with non-vanishing gradient on the set $\{x \, | \, g(x) = 0\}$. Set  
    \begin{equation*}
        D := \left\{x \, \middle| \, g(x) \leq 0 \right\}
    \end{equation*}
    For $x \in D \setminus \partial D$, we have $T_D(x) = \bR^d$ and for $x \in \partial D$, we have
    \begin{equation*}
        T_D(x) = \left\{v \in \bR^d \, \middle| \, \ip{v}{-n(x)} \geq 0 \right\}.
    \end{equation*}
    \end{lemma}

    \begin{proof}
    By proposition 4.1 (a) of \cite{De92}, we find $T_D(x) = \bR^d$ if $x \notin \partial D$. It thus suffices to consider $x \in \partial D$. 
    
    We start with the proof that $T_D(x) \subseteq \left\{v \in \bR^d \, \middle| \, \ip{v}{-n(x)} \geq 0 \right\}$. Thus pick $v \in T_D(x)$. Then there are $t_n, \eps_n \downarrow 0$ and $x_n \in D$ such that
    \begin{equation*}
        \vn{x + t_n v - x_n} \leq \eps_n t_n.
    \end{equation*}
    Taylor expaning $g$ around $x$ yields
    \begin{equation*}
        g(x_n) = g(x) + \ip{\nabla g(x)}{x_n - x} + o(\vn{x_n - x}).
    \end{equation*}
    Rewriting with $g(x) = 0$, $g(x_n) \leq 0$, and adding and subtracting $t_n v$ yields
    \begin{equation*}
        0 \geq t_n \ip{\nabla g(x)}{v} + \ip{\nabla g(x)}{x_n - (x+t_nv)} + t_n o(\eps_n + \vn{v})
    \end{equation*}
    Division by $t_n$ and taking a limit establishes $\ip{\nabla g(x)}{v} \leq 0$, establishing the inclusion.

    For the other inclusion, by closedness of $T_D(x)$ it suffices to consider $v$ for which $\ip{v}{n(x)} < 0$. As $|\nabla g(x)| > 0$ this implies
    \begin{equation*}
        0 > \ip{v}{\nabla g(x)} = \lim_{t \downarrow 0} \frac{g(x+t v) - g(x)}{t}.
    \end{equation*}
    We derive that for small $t$, the line segment $x + t v$ lies in $D$. This implies that
    \begin{equation*}
        \liminf_{t \downarrow 0} \frac{d(x + t v,D)}{t} = 0
    \end{equation*}
    establishing that $v \in T_D(x)$.
    \end{proof}

    \begin{lemma} \label{lemma:curve_has_correct_tangents}
    Consider the context of Lemma \ref{lemma:identification_tangent_cone}. Let $\gamma : [0,T] \rightarrow \bR^d$ be absolutely continuous. Suppose that $\gamma(t) \in D$ for all $t$. If $\gamma$ is differentiable at $t \in [0,T)$, then $\dot{\gamma}(t) \in T_{D}(\gamma(t))$.
    \end{lemma}

\begin{proof}
Fix $t < T$. If $\gamma(t) \notin \partial D$, then there is nothing to prove. So assume $\gamma(t) \in \partial D$. Then, we have
\begin{equation*}
    \ip{\nabla g(\gamma(t)}{\dot{\gamma}(t)} \leq 0
\end{equation*}
as $g \leq 0$ on $D$. As $|\nabla g(\gamma(t))| > 0$, we conclude by Lemma \ref{lemma:identification_tangent_cone} that $\dot{\gamma}(t) \in T_{\gamma(t)} D$. 
\end{proof}

    \subsection{An existence result on closed domains}

	\begin{theorem}[Theorem 2.2.1 in \cite{Ku00}, Lemma 5.1 in \cite{De92}] \label{theorem:solve_differential_inclusion}
		Let $D \subseteq \bR^d$ be closed and let $F : D \rightrightarrows \bR^d$ satisfy
		\begin{enumerate}[(a)]
			\item $F$ has closed convex non-empty values and is upper hemi-continuous;
			\item for every $x$, we have $F(x) \cap T_D(x) \neq \emptyset$;
			\item $F$ has bounded growth: there is some $c > 0$ such that $\vn{F(x)} = \sup\left\{|z| \, \middle| \, z \in F(x) \right\} \leq c(1 + |x|)$ for all $x \in D$.
		\end{enumerate}
		Then the differential inclusion $\dot{\gamma} \in F(\gamma)$ has a solution on $\bR^+$ for every starting point $x \in D$.
	\end{theorem}

\printbibliography
\end{document}